\documentclass[12pt]{amsart}
\usepackage{amssymb}
\usepackage{amsmath}
\usepackage{graphicx}
\newtheorem{theorem}{Theorem}%[section]
\newtheorem{lemma}{Lemma}

\begin{document}
\title[Fej\'er and Suffridge polynomials in the DFC theory]{Fej\'er and Suffridge polynomials in the delayed feedback control theory}
\author{D. DMITRISHIN, A. KHAMITOVA, A. KORENOVSKYI AND A. STOKOLOS}

\begin{abstract}
A remarkable connection between optimal delayed feedback control (DFC) and complex polynomial mappings of the unit disc is established.  The explicit form of extremal polynomials turns out to be related with the Fej\'er polynomials. The constructed DFC can be used to stabilize cycles of one-dimensional non-linear discrete systems.   
\end{abstract}
\maketitle

\section{Introduction}

Non-negative trigonometric polynomials appear in many problems of harmonic analysis, univalent mappings, approximation theory, orthogonal polynomials on torus, number theory and in other branches of mathematics. One of the most natural and beautiful application of the properties of  non-negative trigonometric polynomials occurs solving extremal problems. 

For non-negative trigonometric polynomials $\sum_{j=0}^n a_j\cos jt,\; a_0=1,$ the following Fej\'er inequality holds \cite{F1} (see also \cite{PS}, 6.7, problem 52):
 $$
 |a_1|\le 2\cos\frac\pi{n+2}.
 $$        
 This inequality is sharp and the extremizer 
 $$
 \frac2{n+2}\sin^2\frac\pi{n+2}\Phi_n^{(1)}(t)
 $$       
is unique,  where 
 $$
\Phi_n^{(1)}(t)= \left(\frac{\cos\frac {n+2}{2}t}{\cos t-\cos\frac{\pi}{n+2}} \right)^2.
 $$       
 In 1900 L.Fej\'er showed \cite{F3}  that trigonometric polynomial 
 $$
 1+2\sum_{j=1}^n\left(1-\frac j{n+1}\right)\cos jt       
 $$      
 is nonnegative by proving the multiplicative presentation
 $$
 1+2\sum_{j=1}^n\left(1-\frac j{n+1}\right)\cos jt  = \frac1{n+1}\Phi_n^{(2)}(t),
 $$
 where
 $$
 \Phi_n^{(2)}(t)= \left(\frac{\sin\frac {n+1}{2}t}{\sin\frac t2} \right)^2.
 $$
 In \cite{F4} (see also \cite{PS}, 6.7, problem 50) Fej\'er proved the extremal property of the function $\Phi_n^{(2)}(t),$ that the maximal
 value of the non-negative trigonometric polynomial  $\sum_{j=0}^n a_j\cos jt,\; a_0=1$ does not exceed $n+1.$ Moreover, the equality happens only for the 
 polynomial $\frac1{n+1}\Phi_n^{(2)}(t)$ and only at points $2\pi k, k\in\mathbb Z.$
 
 The functions $ \Phi_n^{(i)}(t), i=1,2$ are called Fej\'er kernels. They possess several extreme properties \cite{D1}. In current paper a new extremal property is established. 
 Namely, we show that these kernels generate polynomial mappings of the unit disc in the complex plane 
 with the maximal size in a certain directions. 
 The coefficients of the extremal polynomials are defined in a unique way. They are linearly connected with the strength coefficients of the optimal control with 
 the delayed feedback which robustly stabilize cycles of the one-dimensional discrete dynamic systems. 
  
\section{Optimal Stabilization and extreme properties of polynomial mappings}
    
Impact on the problem of optimal chaotic regime is 
fundamental in nonlinear dynamics. The purpose of these actions is
synchronization of chaotic motions, or conversely, randomization of regular motions.
Moreover, the allowed values of the control are small which nonetheless
completely change the character of the movement. In this paper we consider the problem of optimal
stabilization of cycles in families of discrete autonomous systems with delayed
feedback control (DFC)  methods \cite{Ch6}.

We are given a scalar nonlinear discrete open-loop system 
\begin{equation} \label{1}
x_{k+1} =f\left(x_{k} \right),\, \, x_{k} \in\mathbb R^{1} ,\, \, n=1,\, \, 2,\, \, \ldots \, \, , 
\end{equation}
with one or more unstable $ T $ - cycles $ (\eta_1, ..., \eta_T), $ where all the numbers $ \eta_1, ..., \eta_T $ are distinct  and 
$ \eta_ {j +1 } = f (\eta_j), j = 1 , ..., T- 1 , \eta_1 = f (\eta_T). $ It is assumed that the multipliers
       $$ 
       \mu = \prod ^ T_ {j = 1 } f '(\eta_j)
       $$
of the  considered unstable cycles are negative. 
       
It is required to stabilize all (or at least some) $T$ - cycles by the control
 \begin{equation} \label{2}
u_{k} =-\sum _{j=1}^{n-1}\varepsilon _{j} \left(f\left(x_{k-jT+T} \right)-f\left(x_{k-jT} \right)\right),\, \, \,  \, 0<\varepsilon _{j} <1,\, \, j=1,\, \, \ldots \, \, ,\, n-1, 
\end{equation} 
 in a way so that the depth of used prehistory $ N ^ \ast = (n- 1 ) T $ would be minimal. \\
 
Note that for state synchronization $ x_k = x_ {k-T} $ the control \eqref{2} vanishes, i.e. closed system takes  
the same form as in the absence of control. 
This means that $ T $ - cycles of the open-loop and closed-loop systems coincide. 
      
 The closed-loop system $ x_ {k +1} = f (x_k) + u_k $ can be written as         
\begin{equation} \label{3}
x_{k+1} =\sum _{j=1}^{n}a_{j} f\left(x_{k-jT+T} \right),                                                             
\end{equation} 
      where
      $$ a_1 = 1 - \varepsilon_1, a_j = \varepsilon_ {j- 1 } - \varepsilon_j, j = 2 , ..., n- 1 , a_n = \varepsilon_ {n- 1} .$$
           It is clear that $ \sum ^ n_ {j = 1 } a_j = 1 .$    
           
         Let apply the following schedule of linearization for constructing of Jacoby matrix of the system \eqref{3}eqref and its characteristic equation. It is clear that
\begin{equation} \label{4a}
\begin{split}
x_{n+1}&=a_1f(x_{n})+a_2f(x_{n-T})...+a_Nf(x_{n-T(N-1)})\\
\dots& \dots \\
x_{n+T}&=a_1f(x_{n+T-1})+a_2f(x_{n-1})...+a_Nf(x_{n-T(N-2)-1})
\end{split}
\end{equation}
Solution to the system \eqref{4a} can be written in the form
\begin{equation} \label{5a}
\begin{split}
x_{Tn}&=\eta_1+u_n^1
\\
\dots& \dots \\
x_{Tn+T-1}&=\eta_T+u_n^T
\end{split}
\end{equation}
Substitute solutions  \eqref{5a} in  \eqref{4a} assuming that in a neigborhood of a cycle the quantities $u_n^1,\dots,u_n^T$ are small.

Let $n=Tm.$ Then 
$$
x_{n+1}=x_{Tm+1}=\eta_2+u_m^2, 
x_{n+2}=x_{Tm+2}=\eta_3+u_m^3,
\dots, x_{n+T}=x_{T(m+1)}=\eta_1+u_m^1
$$
Extracting linear part and taking into account that $\eta_1=f(\eta_2),\dots,\eta_t=f(\eta_1)$ we get
\begin{equation} \label{6a}
\begin{split}
u_m^2&=f^\prime(\eta_1) (a_1u_m^1+...+a_Nu_{m-N+1}^1)\\
u_m^3&=f^\prime(\eta_2) (a_1u_m^2+...+a_Nu_{m-N+1}^2)\\
\dots& \dots \\
u_m^T&=f^\prime(\eta_{T-1}) (a_1u_m^{T-1}+...+a_Nu_{m-N+1}^{T-1})\\
u_{m+1}^1&=f^\prime(\eta_T) (a_1u_m^T+...+a_Nu_{m-N+1}^T)
\end{split}
\end{equation}
Since this is a linear system it's solution can be written as
$$
\left(
\begin{array}{c}
u^1_m\\ \dots \\u^T_m
\end{array}
\right)
=\left(
\begin{array}{c}
c_1\\ \dots \\c_T
\end{array}
\right)\lambda^m,
$$
which after substitution to the system \eqref{6a} leads to a system

{\tiny   
\begin{equation}\label{7a}
\left(
\begin{array}{cccccc} 
{-f'\left(\eta _{1} \right)\cdot p\left(\lambda^{-1} \right)} & {1} & {0} & {\ldots } & {0} & {0} \\ 
{0} & {-f'\left(\eta _{2} \right)\cdot p\left(\lambda^{-1} \right)} & {1} & {\ldots } & {0} & {0} \\ 
{\ldots } & {\ldots } & {\ldots } & {\ldots } & {\ldots } & {\ldots } \\ 
{0} & {0} & {0} & {\ldots } & {-f'\left(\eta _{T-1} \right)\cdot p\left(\lambda^{-1} \right)} & {1} \\ 
{\lambda } & {0} & {0} & {\ldots } & {0} & {-f'\left(\eta _{T} \right)  \cdot p\left(\lambda^{-1} \right)} 
\end{array}\right)
\left(
\begin{array}{c}
c_1\\c_1\\ \dots \\c_{T-1} \\c_T
\end{array}
\right)=\left(
\begin{array}{c}
0\\0\\ \dots \\0\\0
\end{array}
\right)
\end{equation}
}
where 
$p (\lambda^{-1}) = (\alpha_1 +\alpha_2\lambda^{-1} + ... + \alpha_n\lambda^ {-n+1})$.
Standard technique to study the stability of T - cycle is check the location of all the zeros of the determinant of the Jacobian matrix
of the system \eqref{7a}  in the unit disc $\mathbb D$ of complex plane. In this case, the  determinant is equal to
$$ (-1) ^ {{T- 1 }} \lambda + \prod ^ T_ {j = 1 } (-f ^\prime (\eta_j) p (\lambda ^ {-1})),$$
where $\prod_{j=1}^Tf^\prime(\eta_j)=\mu.$
Since we need to consider the stability of not one, but several cycles we should considered 
a family of  characteristic equations
\begin{equation} \label{4}
\left\{\lambda -\mu \cdot \left(p\left(\lambda ^{-1} \right)\right)^{T} =0,\, \, \mu \in \left(-\mu ^{*} ,\, -1\right)\right\},                                                
\end{equation}
where $ \mu ^ \ast $ is the lower bound of T - cycles multipliers, and $ p (1) = 1.$ 
  
For a sufficiently small value $ | \mu |, $ all the roots of $ \lambda + | \mu | \cdot \left (p \left (\lambda ^ {-1}) \right) \right) ^ T = 0 $ lie in a unit disc.
With the increasing of $ | \mu |, $ the roots come out to the boundary of the unit disc. For each coefficient vector $ (a_1, .... a_n),$ the minimum value of $ | \mu |$   that keeps the roots in the unit disc will be denoted by $ \mu_o (a_1, ...,a_n).$ 

 If there are coefficients  $\tilde a_1, ....,\tilde a_n, \tilde a_1+....+ \tilde a_n = 1$ such that  
 $ \mu_o (\tilde a_1, ....,\tilde a_n)> \mu ^ \ast, $ then all equations in the family \eqref{4} corresponding to these
coefficients have roots only in the unit disc and therefore the cycles of the system \eqref{1} can be stabilized
by the control  \eqref{2}. In this regard a number of problems pope out.\bigskip
      
{\bf Problem 1.}  Demonstrate that for any $n>2$ and for any positive integer $T,$ the function 
$ \mu_o (a_1, ...,a_n)$ defined on the hyperplane $a_1+....+a_n = 1$ is semicontinuous from below and
is bounded. Find the value
$$
\mu_n(T)=\sup_{a_1+....+a_n = 1}\mu_o (a_1, ...,a_n).
$$

%$ \mu ^ \ast> 1 $ and  for any positive integer $T$ there is an
%integer $n$ and the coefficients  $a_1, ...., a_n, a_1+....+a_n = 1$ 
%such that $ \mu_o (\alpha_1, ..., \alpha_n)> \mu ^ \ast. $ \bigskip

{\bf  Problem 2.}  Show that for any  $\mu ^ \ast> 1 $ and  for any positive integer $T,$ there is an
integer $n$ and the coefficients $\tilde a_1, ....,\tilde a_n, \tilde a_1+....+ \tilde a_n = 1$ such that
 $ \mu_o (\tilde a_1, ....,\tilde a_n)> \mu ^ \ast,$ i.e. $\mu_n(T)\to\infty$ as $n\to\infty.$\\

%Show that for $n> 2$ and for any  positive integer $T$  the set 
%$$\bigcup_{\sum_{j=1}^n \alpha_j=1} \mu_o (\alpha_1, ..., \alpha_n)$$ is open.  Find 
%$$\sup\left\{ \bigcup_{\sum_{j=1}^n \alpha_j=1} \mu_o (\alpha_1, ..., \alpha_n)\right\}=\mu_n(T).$$

{\bf  Problem 3.} (Dual to Problem 2). For a fixed $T$ and a fixed $ \mu ^ \ast, $ find the minimal positive integer $n$ and strength coefficients $\varepsilon_ 1,\dots, \varepsilon_ {n- 1} $ of the controller (2) such that
all $T$- cycles for the system \eqref{1} closed by this control  with multipliers $ \mu \in (- \mu ^ \ast, 0) $ will be stable. \\

The solution  $n ^ \ast $ to the Problem 3  is related to the solution  $ \mu_n (T) $ of the Problem 2 by the relations
$$ \mu_ {n \ast} (T)> \mu ^ \ast, \quad\mu_ {n ^ \ast- 1 } (T) \le \mu ^ \ast.$$ 

Note that $T$- cycle stability  by itself does not mean practical realization of trajectories of this cycle  for a given control. The trajectories start being attracted to the cycle when the previous coordinate values fall into the basin of attraction of the cycle in the space of initial data.
This basin of attraction may be so small that no segment of trajectories of an open-loop or  closed-loop system would be inside this region. Thus, another problems arises.\\

{\bf  Problem 4.}  For a closed system (3) evaluate the domain of attraction of $T$-cycles in the space of initial data.\\

{\bf  Problem 5.}  For a closed system (3) with stable $T$- cycles  determine all stable cycles of length different from $T$.\\
 
Below we provide solution  to the Problems 1, 2 and 3 for $T = 1,2.$     
     
\section{Construction of the objective function}     
 
 The equation \eqref{4} implies that $\dfrac1\mu=\dfrac1\lambda\left(p\left(\dfrac1\lambda\right)\right)^T,$  where $\lambda\in\mathbb D.$ Therefore the value $\dfrac1\mu$ is in the image of the exterior of the unit disc under the mapping
 $z(p(z))^T.$ The boundary of this image consists entirely of the points in the set $ \left \{e ^ {i \omega} (p (e ^ {i \omega})) ^ T: \omega \in [0,2 \pi) \right \} $ but not necessarily from all those points. It is assumed 
that $\mu$ is real and negative, therefore $\dfrac1{|\mu|}$ {\it  is equal to maximal distance from zero to the the boundary of the image of the exterior of the unit disc in the direction of negative real semi-axis}: 
 $$ 
 \max_ {\omega \in [0,2 \pi)} \left \{\left| p(e ^ {i \omega}) \right|^ T: \arg (e ^ {i \omega} (p(e ^ {i \omega})) ^ T) = \pi \right\}=
 $$  
$$ 
\displaystyle -\min_ {\omega \in [0,2 \pi)} \left \{\Re \left(e ^ {i \omega}p(e ^ {i \omega})) ^ T\right): \Im\left(e ^ {i \omega}p(e ^ {i \omega})) ^ T\right)=0 \right \}.
$$ 
The problem 1 is to minimize the distance:           
$$ 
\frac { 1 } {\mu_n (T)} = \displaystyle-\sup_ {p(z): p ( 1) = 1 } \left \{\displaystyle \min_ {\omega \in [ 0 , 2 \pi)} \left \{\Re (e ^ {i \omega} (p (e ^ {i \omega})) ^ T): 
        \Im (e ^ {i \omega} (p (e ^ {i \omega})) ^ T) = 0 \right \} \right \} =
$$ 
$$ 
\inf_ {p(z): p ( 1) = 1}  \left\{\max_ {\omega \in [ 0 , 2 \pi)} \left\{ \left| p(e ^ {i \omega})\right|^T \arg (e ^ {i \omega} (p (e ^ {i \omega})) ^ T) = \pi\right \} \right \}.
$$    
  
 {\bf   Remark}. Since $ \bar z (p (\bar z)) ^ T =\overline{z (p (z)) ^ T} $, then in the above formulas $ \omega $ can be considered in the interval $ [ 0 , \pi/T]. $ \bigskip
 
By making the change $ \omega = T t, $ the formula for $ \frac { 1 } {\mu_n (T)} $ can be written as 
\begin{equation} \label{7}
\frac{1}{\mu _{n} (T)} =\, \left[\mathop{\inf }\limits_{p(z):p(1)=1} \, \left\{\, \mathop{\max }\limits_{t\in \left[0,\frac{\pi }{T} \right)} \left\{\left|\, p(e^{iT\, t} )\right|:\, \, \, \arg \left(e^{it} \left(p(e^{iT\, t} )\right)\right)=\frac{\pi }{T} \right\}\, \right\}\right]^{T} .                         
\end{equation} 

%$$
%\frac { 1 } {\mu_n (T)} =
%\left[
  % \inf_ {p(z): p ( 1 ) = 1}
      % \left \{
          % \max_ {t \in [ 0 , \frac {\pi} {T})} | p(e ^ {iTt})|: \arg (e ^ {iTt} (p(e ^ {iTt}))) =\frac\pi T
              %   \right \}
%\right]^T.\eqno(7)
%$$
Since 
$$
(ze ^ {i \frac { 2 \pi} {T}}) \cdot p \left(\left(ze ^ {i \frac { 2 \pi} {T}} \right) ^ T\right)
 = e ^ {i \frac { 2 \pi} {T}} z \cdot p (z ^ T),
$$
the polynomial $ z \cdot p(z ^ T) $ has  $ T $-symmetry.
   
  Formula \eqref{7} for $ T = 1 $ and 2 can be conveniently written as
 \begin{equation} \label{8}
\frac{1}{\mu _{n} (1)} =-\mathop{\sup }\limits_{\sum _{j=1}^{n}a_{j}  =1} \left\{\, \mathop{\min }\limits_{\omega \in \left[0,\pi \right]} \left\{\, \sum _{j=1}^{n}a_{j} \cos jt :\, \, \sum _{j=1}^{n}a_{j} \sin jt =0\right\}\, \right\},                         
\end{equation} 
\begin{equation} \label{9}
\frac{1}{\mu _{n} (2)} =\left[\mathop{\inf }\limits_{\sum _{j=1}^{n}a_{j}  } \left\{\, \mathop{\max }\limits_{\omega \in \left[0,\frac{\pi }{2} \right]} \left\{\, \sum _{j=1}^{n}a_{j} \sin (2j-1)t :\, \, \sum _{j=1}^{n}a_{j} \cos (2j-1)t =0\right\}\, \right\}\right]^{2} .            
\end{equation} 

%$$
%\frac { 1 } {\mu_n ( 1 )} = - \sup_ {\sum \limits ^ {n} _ {j = 1 } {\alpha_j = 1} }
%\left \{\min_ {\omega \in [ 0 , \pi]} \left \{\sum \limits ^ {n} _ {j = 1 } \alpha_ {j} \cos {jt}: \sum \limits ^ {n} _ {j = 1 } \alpha _ {j} \sin {jt} = 0 \right \} \right \},
%$$
%$$ 
%\frac { 1 } {\mu_ {n} ( 2 )} = \left [\displaystyle
 %       \inf_{\sum \limits ^ {n} _ {j = 1 } {\alpha_j = 1} }\left \{\max_ {\omega \in [ 0 , \pi/2]} \left \{ \sum ^ {n} _ {j = 1 } \alpha_ {j} \sin ( 2 j- 1 ) t: \sum ^ {n} _ {j = 1 } \alpha_ {j} \cos ( 2 j- 1 ) t = 0 \right\}\right\}\right]^2
%$$ \bigskip 
 Surprisingly, the values of $ \mu_n ( 1 ) $ and $ \mu_n ( 2) $ admit a simple explicit expression
through $n,$ and the coefficients of limiting extreme polynomials  $ z \cdot p ^ {( 1) } _ {0 } (z), z \cdot (p ^ {( 2) } _ {0 } (z)) ^ 2 $ 
can be easily expressed in terms of the Fejer kernel $ \Phi ^ {(1)} _ {n} (t)$ and $ \Phi^{(2)}_n (t) .$  

\section{Auxiliary results}

The main idea of solving the above problems is to determine
necessary conditions for the mappings of the form $ z \cdot p (z ^ T) $
maximally reduce the class of possible mappings %Namely, it will be shown that with the need
%image of the unit circle must not have self-intersections , ie extreme mapping should be sheeted. 

The self-intersection of the circle image corresponds to a factorization of the conjugate trigonometric polynomials. The factorization theorem can be regarded as real analogue of Bezout's theorem

\begin{theorem}\label{trm1}
Let
\begin{equation} \label{10}
C(t)=\sum _{j=1}^{n}a_{j} \cos jt ,  \qquad S(t)=\sum _{j=1}^{n}a_{j} \sin jt  
\end{equation}
 be a pair of conjugate trigonometric polynomials with real coefficients. And let the equalities
 \begin{equation} \label{11}
S(t_{1} )=\, \, \ldots \, \, =S(t_{m} )=0, \qquad  C(t_{1} )=\, \, \ldots \, =C(t_{m} )\, =\gamma ,                                           
\end{equation} 
are valid where the numbers $ t_1 ..., t_m $ belong to the interval $ (0 , \pi) $, and $ 2m \le n$.
Then trigonometric polynomials \eqref{10} admit  a presentation
\begin{equation} \label{12}
C(t)=\gamma +\prod _{j=1}^{m}(\cos t-\cos t_{j} )  \sum _{k=m}^{n-m}\alpha _{k}  \cos kt,\; S(t)=\prod _{j=1}^{m}(\cos t-\cos t_{j} ) \sum _{k=m}^{n-m}\alpha _{k}  \sin kt,            
\end{equation}
where $ \alpha_ {m} = -2 ^ {m} \gamma$ and  the coefficients $ \alpha_ {m}, .... \alpha_ {n-m} $ can be uniquely expressed in terms of $\gamma, a_1, ...., a_ {n}.$ 
\end{theorem}

\begin{proof} Consider the algebraic polynomial
$$ F (z) = - \gamma + \sum \limits ^ {n} _ {j = 1 } \alpha_ {j} z ^ {j}.$$
Then
$$
C (t) = \gamma + \Re \left \{F (e ^ {it}) \right \},
S (t) = \Im \left \{F (e ^ {it}) \right \}.$$
By \eqref{11} $ F (e ^ {it_ j}) = 0$ and $ F (e ^ {-it_ j}) = 0 , j = 1 , ...., m. $ 
The fundamental theorem of algebra implies the existence of the numbers $ \beta_ { 1 }, ...., \beta_ {n -2m} $ such that
$$
F (z) = \left (
\prod \limits ^ {m} _ {j = 1 } (z-e ^ {it_j}) (z-e ^ {-it_j})
\right)
\left (
- \gamma + \sum ^ {n-2m} _ {k = 1 } \beta_ {k} z ^ {k}
\right).
$$
Let us modify the product
$$
\prod ^ {m} _ {j = 1 } \left (z-e ^ {it_j}\right) \left (z-e ^ {-it_j}\right) =
\prod ^ {m} _ {j = 1 } \left (z ^ 2 -2z \cos t_ {j} +1 \right) = 2 ^ {m} z ^ {m} \prod ^ {m} _ { j = 1 } \left (\frac { 1} {2 } \left (z + \frac { 1 } {z} \right) - \cos t_ {j} \right).
$$
Therefore,
$$
F (z) = \prod ^ {m} _ {j = 1 } \left (\frac { 1} {2 } \left (z + \frac { 1 } {z} \right) - \cos t_ {j} \right) \left (-2 ^ {m}
\gamma z ^ {m} +2 ^ {m} \sum ^ {n-2m} _ {k = 1 } \beta_ {k} z ^ {m +k} \right).
$$
Then
$$
\Re \left \{F \left (e ^ {it} \right) \right \} = \prod ^ {m} _ {j = 1 } \left (\cos t-\cos t_ {j} \right) \left (-2 ^ {m} \gamma \cos {mt} +2 ^ {m} \sum ^ {n-m} _ {k = m +1} \beta_ {k-m} \cos {kt} \right),
$$
$$
\Im \left \{F \left (e ^ {it} \right) \right \} = \prod ^ {m} _ {j = 1 } \left (\cos t-\cos t_ {j} \right) \left (-2 ^ {m} \gamma \sin {mt} +2 ^ {m} \sum ^ {n-m} _ {k = m +1} \beta_ {k-m} \sin {kt} \right),
$$
which implies \eqref{12} with  $ \alpha_ {m} = -2 ^ {m} \gamma, \alpha_ {m + k} = 2 ^ {m} \beta_ {k}, k = 1, ..., n-2m. $
\end{proof}

Formula \eqref{12} allows us to generalize well-known identities, and to obtain new.
\bigskip

{\bf Examples.} 
 
 1. The following multiplicative representations of conjugate Dirichlet kernels are well known 
$$
\sum ^ {m} _ {j = 1 } \sin 2 {jt} = \frac {\sin {mt}} {\sin {t}} \sin (m +1) t, \sum ^ {m} _ {j = 1 } \cos 2 {jt} = \frac {\sin {mt}} {\sin {t}} \cos (m +1) t.
$$
Formula \eqref{12} gives another multiplicative representations of Dirichlet kernels
$$
\sum ^ {m} _ {j = 1 } \sin 2 {jt} = 2 ^ {m} \prod ^ {m} _ {j = 1 } (\cos {t} - \cos \frac {\pi {j}} {m +1}) \sin {mt},
$$
$$
\sum ^ {m} _ {j = 1 } \cos 2 {jt} = -1 +2 ^ {m} \prod ^ {m} _ {j = 1 } (\cos {t} - \cos \frac {\pi j} {m +1}) \cos {mt}.
$$
Letting $t = 0,$ we obtain the identity
$$
\frac {m +1} { 2 ^ {m}} = \prod ^ {m} _ {j = 1} \left(1 - \cos \frac {\pi j} {m +1}\right).
$$

2. In a similar way, one can get multiplicative formulas for the Dirichlet kernels

$$
\sum ^ {2 {m}} _ {j = 1 } \sin {jt} = 2 ^ {m} \prod ^ {m} _ {j = 1 } (\cos t-\cos \frac { 2 \pi j} {2 {m +1}}) \sin {mt},
$$

$$
\sum ^ {2 {m}} _ {j = 1 } \cos {jt} = -1 +2 ^ {m} \prod ^ {m} _ {j = 1 } (\cos t-\cos \frac {2 \pi j} {2 {m +1}}) \cos {mt},
$$

$$
\sum ^ {2 {m +1}} _ {j = 1 } \sin {jt} = 2 ^ {m} \prod ^ {m} _ {j = 1 } (\cos t-\cos \frac { \pi j} {m +1}) (\sin {mt} + \sin (m +1) t),
$$

$$
\sum ^ {2 {m +1}} _ {j = 1 } \cos {jt} = -1 +2 ^ {m} \prod ^ {m} _ {j = 1 } (\cos t-\cos \frac {\pi j} {m +1}) (\cos {mt} + \cos (m +1) t).
$$

3 . Since
$$
(1 + e ^ {-it}) ^ {m} = 1 + \sum ^ {m} _ {j = 1 } {m \choose j} e ^ {-jt},
(1 + e ^ {-it}) ^ {m} = 2 ^ {m} e ^ {-i \frac {mt} { 2}}
\left (
\frac {e ^ {-i \frac {t} { 2 }} + e ^ {i \frac {t} { 2}} } {2}
\right) ^ m,
$$
 we have multiplicative representations
$$
\sum ^ {m} _ {j = 1 }  {m\choose j}\sin2 {jt} = 2 ^ {m} \cos ^ {m} t \sin {mt},
$$
$$
\sum ^ {m} _ {j = 1 }  {m\choose j} \cos2 {jt} = -1 +2 ^ {m} \cos ^ {m} t \cos { mt},
$$
$$
\sum ^ {2m} _ {j = 1 } (-1) ^ {j + m}  {2m\choose j} \sin {jt} = 2 ^ {m} ( 1 - \cos t) ^ {m} \sin {mt},
$$
$$
\sum ^ {2m} _ {j = 0 } (-1) ^ {j + m} {2m\choose j}  \cos {jt} = 2 ^ {m} ( 1 - \cos t) ^ {m} \cos mt.
$$

As a consequence, we obtain the following representations for the Chebyshev polynomials
$$
T_ {m} (x) = \frac { 1} {2 ^ {m} x ^ {m}}
\sum ^ {m} _ {j = 0 }{m\choose j}  T_ {2j} (x),
$$

$$
T_ {m} (x) = \frac { 1} {2 ^ {m} ( 1 -x) ^ {m}} \sum ^ {2m} _ {j = 0 } (-1) ^ {j + m } {2m\choose j}  T_ {j} (x).
$$

In the future, we will need a property of nonlocal separation from zero for the image of
the unit disc. A good example of such claim could is a famous  K\"obe Quarter Theorem \cite{K}.
Unfortunately, it is unclear how to apply it for two reasons. First, there is a different normalization of the mappings: 1 mps 1. Second, considered mappings are not necessarily univalent. 

We can prove a weaker statement which is enough for our needs.

\begin{lemma}\label{lm1} Let
$
F (z) = a_ { 1 } z + ... + a_ {n} z ^ {n}$,  $(a_ {j} \in \mathbb C, j = 1 , ..., n),
$
and $\mathbb D = \left \{| z | < 1 \right \}.$ Then the set $ F (\mathbb D) $ contains a disc with the center at the origin and radius
$$
\frac { 1} {2 ^ n} \sum ^ {n} _ {j = 1 } | a_ {j} |.
$$
\end{lemma}

\begin{proof}  Let $\gamma$ be an exceptional value for the polynomial $F(z)$ for all $z\in \mathbb D.$ Since  $F(z)\le \sum_{j=1}^n |a_j|,$ such values do exist. %$\gamma\not\in  F (\mathbb D).$ 
Then the polynomial $F(z)-\gamma$ does not have roots inside 
$\mathbb D$ and therefore the inversion produces a polynomial $z^n(-\gamma+F(1/z))$ that has all zeros in $\mathbb D,$ i.e. Schur stable. It can be written as 
$$
-\gamma z^n+a_1z^{n-1}+\dots+a_n=-\gamma(z^n-\frac{a_1}\gamma z^{n-1}-\dots-\frac{a_n}\gamma).
$$
Applying Vieta's theorem to the polynomial in parenthesis, we get the estimate 
$\left| \frac{a_j}\gamma\right|\le{n\choose j},\;j=1,\dots,n,$ which implies % \Longrightarrow
$\sum ^ {n} _ {j = 1 } \left| \frac{a_ {j}}\gamma\right| \le 2 ^ {n} -1,$
and finally
$
|\gamma|\ge  \dfrac{1}{2 ^ {n} -1}\sum ^ {n} _ {j = 1 } |a_ {j}|.
$
Note that the above estimate cannot be improved as the example of  $F(z)=(z+1)^n-1$ demonstrates.
 
\end{proof}

\section{Solution of optimization problems}

\subsection{T=1 case}

 \begin{theorem}\label{trm2}
 Let $C(t)$ and $S(t)$  be a pair of conjugated trigonometric polynomials
 $$
 C(t)=\sum_{j=1}^na_j\cos jt,\quad S(t)=\sum_{j=1}^na_j\sin jt,
 $$
normalized by the conditions $\sum\limits_{j=1}^na_j=1$.
Let  $J_1$  be a solution to the extremal problem
 $$\sup\limits_{a_1,\dots,a_n}\min\limits_t\left\{C(t):\ S(t)=0\right\}.$$
 Then
 $$
 J_1=-\tan^2\frac\pi{2(n+1)}.
 $$
 \end{theorem}
\begin{proof} 
Note that $S(\pi)=0$ for any choice of the coefficients $a_1,\dots,a_n.$ Denote
$$
\rho\left(a_1,\dots,a_n\right)=\min_{t\in[0,\pi]}\left\{C(t): S(t)=0\right\}.
$$
The quantity $\sup\{\rho\left(a_1,\dots,a_n\right) \}$ will be evaluated along the set 
$$
A_R=\left\{\left(a_1,\dots,a_n\right):\
\sum\limits_{j=1}^n a_j=1,\,\sum\limits_{j=1}^n\left|a_j\right|\le R\right\},
$$
where $R$ is a large enough positive number. 

Together with  $\rho\left(a_1,\dots,a_n\right)$ we   consider a function
$$
\rho_1\left(a_1,\dots,a_n\right)=\min_{t\in[0,\pi]}\left\{C(t):\ t\in \mathcal T\cup\{\pi\}\right\},
$$
where $\mathcal T$ is a subset of $(0,\pi)$, such that the function $S(t)$ changes sign.
Lemma \ref{lm1} implies that the curve $(C(t),S(t))$ enclosed zero, therefore there is a negative value of the variable $t$ where the trigonometric polynomial $S(t)$ changes the sign, hence $\rho_1\left(a_1,\dots,a_n\right)<0$ for any $a_1,\dots,a_n.$

\begin{lemma}\label{opt}
There is a pair of trigonometric polynomials $\left\{C^0(t),\,S^0(t)\right\}$  such that
$$
\sup\limits_{\left(a_1,\dots,a_n\right)}
\left\{\rho_1\left(a_1,\dots,a_n\right)\right\}
=\min_{t\in[0,\pi]}\left\{C^0(t):\ t\in\mathcal T^0\cup\{\pi\}\right\}
$$
where $\mathcal T^0$ is a subset of the interval $(0,\pi)$, such that the function $S^0(t)$ changes the sign.
\end{lemma}

\begin{proof} The function $\rho_1\left(a_1,\dots,a_n\right)$ is upper semi-continuous on the set $A_R$, besides possibly the points $(a_1,\dots,a_n)$ for which the minimal value of $C(t)$ is achieved in those zeros of the function $S(t),$ where   $S(t)$ does not change the sign.
An upper limit of the function $\rho_1\left(a_1,\dots,a_n\right)$ is equal the value of the function at the points of discontinuity which means that 
the function $\rho_1\left(a_1,\dots,a_n\right)$ is semi-continuous from above. Therefore, it achieves the maximum value on the set $A_R$  
$$
\overline{\rho_1}=\max\limits_{\left(a_1,\dots,a_n\right)\in A_R}
\left\{\rho_1\left(a_1,\dots,a_n\right)\right\}.
$$
It is clear that  $\overline{\rho_1}\ge-1,$ therefore $|\overline{\rho_1}|<1$ and by Lemma \ref{lm1} for
the extremal point  $a^0_1,\dots,a^0_n$ we have the estimate $\sum_{j=1}^n|a_j|\le 2^n.$ 
Thus, for $R>2^n$ the maximum is achieved in an internal point of the region $A_R$ and the optimal pair $\left\{C^0(t),\,S^0(t)\right\}$ that provides the maximum is independent on $R$ if $R>2^n.$ 
\end{proof}

The set  $\mathcal T\cup \left\{\pi \right\}$  is a subset of all zeros of the function $S(t),$ therefore  $\overline{\rho }\le \overline{\rho _{1} }$, where $\overline{\rho }=\mathop{\sup }\limits_{(a_{1} ,\, \ldots \, .,a_{n} )\in A_{R} } \left\{\, \, \rho (a_{1} ,\, \ldots \, ,a_{n} )\, \, \right\}$. A pair of trigonometric polynomials  $\left\{C^{0} (t),S^{0} (t)\right\}$, where the value
 $\overline{\rho _{1} }$ is achieved will be called optimal. 

\begin{lemma}\label{pos}
If polynomial $S(t)$ has zero in $(0,\pi)$ it cannot be optimal.
\end{lemma}
 
\begin{proof}
Let for the optimal polynomial $S^0(t)$ the set 
$$
\mathcal T=\left\{t_1,\dots,t_q\right\}, \; 0\le q\le n-1,
$$
be nonempty. And let
$$
\min\left\{C^0\left(t_1\right),\dots,C^0\left(t_q\right)\right\}=C^0\left(t_1\right),
$$
assuming additionally that
$$
C^0\left(t_1\right)=C^0\left(t_j\right), j=1,\dots,m\; (1\le m\le q)
$$
and
$$
C^0\left(t_1\right)<C^0\left(t_j\right), j=m+1,\dots, q.
$$

Then three cases are possible: either $C^0\left(t_1\right)<C^0(\pi),$ or
$C^0(\pi)\le C^0\left(t_1\right)\le 1,$ or $ C^0\left(t_1\right)>1.$\\

{\it Case 1.}
By Theorem~\ref{trm1},
 the trigonometrical polynomials $S^0(t)$, $C^0(t)$ can be written as
$$
S^0(t)=\prod_{j=1}^m\left(\cos t-\cos t_j\right)
\sum_{k=m}^{n-m}\alpha_k\sin kt,
$$
$$
C^0(t)=-\frac{\alpha_m}{2^m}+\prod_{j=1}^m\left(\cos t-\cos t_j\right)
\sum_{k=m}^{n-m}\alpha_k\cos kt.
$$
Since  $C^0\left(t_1\right)=-\frac{\alpha_m}{2^m}\le 0$ then $\alpha_m>0$.
Moreover,  $C^0(0)=1$, therefore
$$
-\frac{\alpha_m}{2^m}+
\prod\limits_{j=1}^m\left(1-\cos t_j\right)\sum_{k=m}^{n-m}\alpha_k=1,
\quad
\sum_{k=m}^{n-m}\alpha_k=\frac{1+\frac{\alpha_m}{2^m}}{\prod\limits_{j=1}^m\left(1-\cos t_j\right)}>0.
$$
Let us define the following auxiliary polynomials
\begin{multline}\nonumber
\begin{split}
S\left(\theta_1,\dots,\theta_m;t\right)&=N\left(\theta_1,\dots\theta_m\right)\cdot\\
&\cdot\prod_{j=1}^m\left(\cos t-\cos \theta_j\right)\sum_{k=m}^{n-m}\alpha_k\sin kt,\\
C\left(\theta_1,\dots,\theta_m;t\right)&=N\left(\theta_1,\dots\theta_m\right)\cdot\\
&\cdot \left(-\frac{\alpha_m}{2^m}+
\prod_{j=1}^m\left(\cos t-\cos \theta_j\right)\sum_{k=m}^{n-m}\alpha_k\cos kt\right),
\end{split}
\end{multline}
where the normalizing factor $N\left(\theta_1,\dots,\theta_m\right)$ guaranties the sums of the coefficients of each polynomials $S\left(\theta_1,\dots,\theta_m;t\right)$ and $C\left(\theta_1,\dots,\theta_m;t\right)$ to be 1. For the polynomial  $S\left(\theta_1,\dots,\theta_m;t\right),$ the set of sign changes will be
$\mathcal T_\theta=\left\{\theta_1,\dots,\theta_m,t_{m+1},\dots,t_q\right\}$. It is clear that
$S\left(t_1,\dots,t_m;t\right)\equiv S^0(t)$ and
$C\left(t_1,\dots,t_m;t\right)\equiv C^0(t)$.
The factor $N\left(\theta_1,\dots,\theta_m\right)$
is determined by the condition
$C\left(\theta_1,\dots,\theta_m;0\right)=1$, i.e.
$$
N\left(\theta_1,\dots,\theta_m\right)=
\frac1{-\frac{\alpha_m}{2^m}+\prod\limits_{j=1}^m\left(1-\cos\theta_j\right)
\sum\limits_{k=m}^{n-m}\alpha_k}.
$$
So, the polynomials
$S\left(\theta_1,\dots,\theta_m;t\right)$ and
$C\left(\theta_1,\dots,\theta_m;t\right)$
could be defined by the expressions
$$
S\left(\theta_1,\dots,\theta_m;t\right)=
\frac{\prod\limits_{j=1}^m\left(\cos t-\cos \theta_j\right)\sum\limits_{k=m}^{n-m}\alpha_k\sin kt}%
{-\frac{\alpha_m}{2^m}+\prod\limits_{j=1}^m\left(1-\cos\theta_j\right)
\sum\limits_{k=m}^{n-m}\alpha_k},
$$
$$
C\left(\theta_1,\dots,\theta_m;t\right)=
\frac{-\frac{\alpha_m}{2^m}+
\prod\limits_{j=1}^m\left(\cos t-\cos \theta_j\right)\sum\limits_{k=m}^{n-m}\alpha_k\cos kt}%
{-\frac{\alpha_m}{2^m}+\prod\limits_{j=1}^m\left(1-\cos\theta_j\right)\sum\limits_{k=m}^{n-m}\alpha_k}.
$$
Note that the coefficients $\alpha_m,\dots, \alpha_{n-m}$ are independent of the choice of parameters
$\theta_1,\dots,\theta_m.$ 

Let us show that for some $\theta_1,\dots,\theta_m$  the value of $\rho_1$ for the pair
$$
\left\{C\left(\theta_1,\dots,\theta_m;t\right),\, S\left(\theta_1,\dots,\theta_m;t\right)
\right\}
$$
is bigger then for the pair  $\left\{C^0(t), S^0(t),\right\}$,
i.e. the pair $\left\{C^0(t), S^0(t),\right\}$  cannot be optimal.
\medskip

By the defintion of $C\left(\theta_1,\dots,\theta_m;\theta_j\right)$  we get
\begin{equation}\nonumber%\label{*}%{multline}\nonumber
\begin{split}
C\left(\theta_1,\dots,\theta_m;\theta_1\right)=\dots
=C\left(\theta_1,\dots,\theta_m;\theta_m\right)=\\
\frac{-\frac{\alpha_m}{2^m}}
{-\frac{\alpha_m}{2^m}+\prod\limits_{j=1}^m\left(1-\cos\theta_j\right)\sum\limits_{j=m}^{n-m}\alpha_j}.
\end{split}
\end{equation}%{multline}

Since  $\alpha_m>0$ and $\sum\limits_{j=m}^{n-m}\alpha_j>0$ ,
then $C\left(\theta_1,\dots,\theta_m;\theta_1\right)$, $j=1,\dots,m$,
is an increasing function of the parameters $\theta_1,\dots,\theta_m$.

Let $0<\theta_j-t_j<\varepsilon, j=1,\dots, m.$ From the continuity of trigonometric polynomials on $t$ and an all coefficients the following
inequalities hold
$$
C\left(\theta_1,\dots,\theta_m;\theta_j\right)>C^0\left(t_j\right),\; j=1,\dots,m
$$
$$
\left|C\left(\theta_1,\dots,\theta_m;t_j\right)-C^0\left(\theta_j\right)\right|<\delta,\; j=m+1,\dots,q,
$$
$$
\left|C\left(\theta_1,\dots,\theta_m;\pi\right)-C^0(\pi)\right|<\delta
$$ 
for any  $\delta$ with appropriate choice of  $\varepsilon.$ The above inequalities mean that the value 
$$
\min\left\{C\left(\theta_1,\dots,\theta_m;\theta_1\right),\dots,
C\left(\theta_1,\dots,\theta_m;\theta_m\right),
C\left(\theta_1,\dots,\theta_m;t_{m+1}\right),\dots,\right.
$$
$$\left.
C\left(\theta_1,\dots,\theta_m;t_q\right),
C\left(\theta_1,\dots,\theta_m;\pi\right)\right\}
$$
is larger then
$\min\left\{C^0\left(t_1\right),\dots,C^0\left(t_q\right),C^0(\pi)\right\}$
at least for small enough positive
$\theta_j-t_j$, $j=1,\dots,m$, i.e. the pair
$\left\{C^0(t), S^0(t),\right\}$  is not an optimal.\\

{\it Case 2.}
Let compute
$$
C^0(\pi)=-\frac{\alpha_m}{2^m}-\prod_{j=1}^m\left(1+\cos t_j\right)
\sum_{j=m}^{n-m}(-1)^{j+m-1}\alpha_j,
$$
$$
C\left(\theta_1,\dots,\theta_m;\pi\right)=
-\frac{\frac{\alpha_m}{2^m}+\prod\limits_{j=1}^m\left(1+\cos\theta_j\right)
\sum\limits_{j=m}^{n-m}(-1)^{j+m-1}\alpha_j}%
{-\frac{\alpha_m}{2^m}+\prod\limits_{j=1}^m\left(1-\cos\theta_j\right)
\sum\limits_{j=m}^{n-m}\alpha_j},
$$
and
$$C\left(t_1,\dots,t_m;\pi\right)=C^0(\pi).$$

Since we assume that  $C^0(\pi)\le-\frac{\alpha_m}{2^m}$,
then   $\sum\limits_{j=m}^{n-m}(-1)^{j+m-1}\alpha_j\ge0$,
and the quantity
$
\frac{\alpha_m}{2^m}+\prod\limits_{j=1}^m\left(1+\cos \theta_j\right)
\sum\limits_{j=m}^{n-m}(-1)^{j+m-1}\alpha_j
$
is decreasing with respect to each parameter  $\theta_1,\dots,\theta_m$. 

Since by the assumption $C^0(t_1)=-\frac{\alpha_m}{2^m}\le 1$ then
$$
\sum_{j=m}^{n-m}\alpha_j=\frac{1+\frac{\alpha_m}{2^m}}{\prod\limits_{j=1}^m\left(1-\cos\theta_j\right)}>0.
$$
For small increments of
$\theta_j-t_j$, $j=1,\dots,m$ the value of 
$
-\frac{\alpha_m}{2^m}+\prod\limits_{j=1}^m\left(1-\cos \theta_j\right)
\sum\limits_{j=m}^{n-m}\alpha_j
$
is close to 1, and is increasing with respect to each parameter $\theta_1,\dots,\theta_m$.
Therefore $C\left(\theta_1,\dots,\theta_m;\pi\right)$ is increasing with respect to each parameters $\theta_1,\dots,\theta_m$.
At the same time $C\left(\theta_1,\dots,\theta_m;\theta_j\right)$, $j=1,\dots,m$ are equal and are increasing
with respect to each parameters  $\theta_1,\dots,\theta_m$ too. Therefore in this case the pair
$\left\{C^0(t), S^0(t)\right\}$  cannot be an optimal either.\\

{\it Case 3.} By the  Theorem 1, the optimal polynomials can be written in the form
$$
C^0(t)=-\frac{\beta_1}2+\left(\cos t-\cos t_1\right)\sum_{j=1}^{n-1}\beta_j\cos jt ,
$$
$$
S^0(t)=\left(\cos t-\cos t_1\right)\sum_{j=1}^{n-1}\beta_j\sin jt
$$
where 
$$
-\frac{\beta_1}2+\left(1-\cos t_1\right)\sum_{j=1}^{n-1}\beta_j=1, \quad 
C^0(t_1)=-\frac{\beta_1}2>1
$$
and
$$
\quad C^0(\pi)= -\frac{\beta_1}2-\left(1+\cos t_1\right)
\sum_{j=1}^{n-1}(-1)^j\beta_j <0.
$$
Therefore
$$
\sum_{j=1}^{n-1}\beta_j=\frac{\frac{\beta_1}2+1}{1-\cos t_1}<0,\quad \sum_{j=1}^{n-1}(-1)^j\beta_j>0.
$$
Consider a family
$$
C(\theta,t)=\frac{-\frac{\beta_1}2+\left(\cos t-\cos\theta\right)\sum_{j=m}^{n-1}\beta_j\cos jt }
{-\frac{\beta_1}2+\left(1-\cos\theta\right)\sum_{j=1}^{n-1}\beta_j},
$$
$$
S(\theta,t)=\frac{\left(\cos t-\cos\theta\right)\sum_{j=1}^{n-1}\beta_j\sin jt }
{-\frac{\beta_1}2+\left(1-\cos\theta\right)\sum_{j=1}^{n-1}\beta_j}.
$$
It is clear that $C(t_1,t)\equiv C^0(t)$ and $S(t_1,t)\equiv S^0(t).$

Let compute
$$
C(\theta,\pi)=\frac{-\frac{\beta_1}2-\left(1+\cos\theta\right)\sum_{j=1}^{n-1}(-1)^j\beta_j}
{-\frac{\beta_1}2+\left(1-\cos\theta\right)\sum_{j=1}^{n-1}\beta_j}.
$$
The function $C(\theta,\pi)$ is either monotonic on $(0,\pi)$ as a rational function of $\cos\theta$ or is identically constant. 

If it is monotonic, then there exists $\theta_1$ such that $C(\theta_1,\pi)>C(t_1,\pi)=C^0(\pi).$
Thus, the pair $\left\{ S^0(t), C^0(t) \right\}$ is not the optimal.

Let now $C(\theta,\pi)$ be a constant. Since $C(t_1,\pi)=C^0(\pi)<0$ then  $C(\theta,\pi)<0.$ Therefore
$$
\lim_{\theta\to\pi}C(\theta,\pi)=\frac{-\frac{\beta_1}2}{-\frac{\beta_1}2+2\sum_{j=1}^{n-1}\beta_j}<0.
$$
Since $-\dfrac{\beta_1}2>0,$ the function $-\frac{\beta_1}2+\left(1-\cos \theta\right)\sum_{j=1}^{n-1}\beta_j$ is positive at zero and negative for $\theta$ close to $\pi,$ because the denominator of the above fraction is a value of the function at $\pi.$ Therefore, there is zero value and then there exists $\theta_2$ such that 
$$
C(\theta_2,\theta_2)=\frac{-\frac{\beta_1}2}{-\frac{\beta_1}2+(1-\cos\theta_2)\sum_{j=1}^{n-1}\beta_j}>2^n.
$$
Since the absolute value of a trigonometric polynomial does not exceed the sum of the absolute values of the coefficients, then the sum of the absolute values of the coefficients of the polynomial
$C(\theta_2,t)$ is bigger then $2^n.$ 
%By the Lemma 1 $$\min_t\{C(\theta_2,t):t\in\mathcal T_{\theta_2}\}<-1,$$
%zwhere $\mathcal T_{\theta_2}=\{\theta_2,t_2,\dots,t_q\}$ is the set of sigh change for the polynomial
%$C(\theta_2,t).$ That does mean that some of $C(\theta_2,t_j),\;j=2,\dots,q$ become less then 1.
 %for $\theta\in(t_1,\theta_2).$ 
 This means that either some of the functions $C(\theta;t_j), j=2,\ldots,q,$ become less then one for certain  $\theta\in(t_1,\theta_2)$ being positive,
 or Lemma 1 implies that  $C(\theta_2,\pi)<-1.$ In any of those cases the quantities  $C(\theta_2,\pi)=C(t_1,\pi)=C^0(\pi)$ cannot be extremal. 

Thus, it is shown that the set $\mathcal T$ is empty, therefore for the optimal pair  $\left\{C^0(t),\,S^0(t)\right\}$ we have $S^0(t)>0$, $t\in(0,\pi).$ The lemma is proved.
\end{proof}

\medskip

Trigonometric polynomial $S^0(t)$ can be written as
$$
S^0(t)=\sin t\cdot\left(\gamma^0_1+2\gamma^0_2\cos t+\dots+2\gamma^0_n\cos(n-1)t\right),
$$
where $\gamma^0_s=\sum a^0_j$ and the summation runs over indices  $s\le j\le n$ of same parity 
with $s,\, s=1,...,n.$ There is a bijection between $a_1,\dots,a_n$ and $\gamma_1,\dots,\gamma_n.$ The normalization $a_1+...+a_n=1$ implies $\gamma^0_1+\gamma^0_2=1.$

Since $\mathcal T=\emptyset$ then
\[
\overline{\rho }_{1} =\mathop{\max }\limits_{(a_{1} ,...,a_{n} )\in A_{R} } \left\{\, C(\pi ): S(t)>0, \,t\in(0,\pi) \right\}=
\]
\[
\mathop{\max }\limits_{(a_{1} ,...,a_{n} )\in A_{R} } \left\{-a_{1} +a_{2} -a_{3} +... : S(t)>0, \,t\in(0,\pi)\right\}.
\] 
Note that $-a_{1} +a_{2} -a_{3} +\, \, \ldots =-\gamma _{1} +\gamma _{2} $.

Since ${S^0(t)}/{\sin t}$  is non-negative, the well-known Fej\'er inequality for non-negative polynomials ~\cite{F1}
(see also~\cite{PS} 6.7, Problem 52) implies that
$$
\left|\gamma_2\right|\le\cos\frac\pi{n+1}\cdot\left|\gamma_1\right|.
$$
From here we get
$$
\overline{\rho }_{1}\le \overline{\rho }_{2}:= \max_{\gamma_1,\gamma_2}\left\{-\gamma_1+\gamma_2:\ \gamma_1+\gamma_2=1,\,
\left|\gamma_2\right|\le\cos\frac\pi{n+1}\cdot\left|\gamma_1\right|\right\}.
$$
The conditional maximum is achieved for 
$$
\gamma_1^0=\frac1{1+\cos\frac\pi{n+1}},\quad
\gamma_2^0=\frac{\cos\frac\pi{n+1}}{1+\cos\frac\pi{n+1}},
$$
and is equal to
$$
\overline{\rho }_{2}=-\frac{1-\cos\frac\pi{n+1}}{1+\cos\frac\pi{n+1}}=-\tan^2\frac\pi{2(n+1)}.
$$
The polynomial  ${S^{0} (t)}/{\sin t} =\gamma _{1}^{0} +2\gamma _{2}^{0} \cos t+...+2\gamma _{n}^{0} \cos (n-1)t$ is a non-nagative Fej\'er polynomial and its coefficients 
$a_{1}^{0} ,\, \ldots \, ,a_{n}^{0} $ are determined in a unique way: $a_{1}^{0} =\gamma _{1}^{0} -\gamma _{3}^{0} ,\; a_{2}^{0} =\gamma _{2}^{0} -\gamma _{4}^{0} ,\; a_{3}^{0} =\gamma _{3}^{0} -\gamma _{5}^{0} \, \, \ldots $. Since $a_{j}^{0} $ are positive then $\sum _{j=1}^{n}\left|a_{j}^{0} \right|= \sum _{j=1}^{n}a_{j}^{0} = 1$  which is independent of $R.$ Therefore for any
$a_{1} ,\, \ldots \, \, ,\, \, a_{n} $  such that  $\sum _{j=1}^{n}\left|a_{j} \right| =1,$ the following inequalities are valid
\[\rho _{1} (a_{1} ,\, \ldots \, ,a_{n} )\le \overline{\rho _{1} },\quad \rho (a_{1} ,\, \ldots \, ,a_{n} )\le \overline{\rho }\le \overline{\rho _{1} }.\] 

To prove that the function $\rho (a_{1} ,\, \ldots \, ,a_{n} )$ the supremum is achieved and is equal to $\overline{\rho _{1} },$ let consider one-parametric family of trigonometric polynomials
$$
S^\varepsilon(t)=\frac{a_1^0+\varepsilon}{1+\varepsilon}\sin t+
\frac{a_2^0}{1+\varepsilon}\sin 2t+\dots+
\frac{a_n^0}{1+\varepsilon}\sin nt.
$$
It is clear that $\frac{a_1^0+\varepsilon}{1+\varepsilon}+ \frac{a_2^0}{1+\varepsilon}+\dots+
\frac{a_n^0}{1+\varepsilon}=1$ and
$S^\varepsilon(t)=\frac{S^0(t)}{1+\varepsilon}+\frac\varepsilon{1+\varepsilon}\sin t$,
$C^\varepsilon(t)=\frac{C^0(t)}{1+\varepsilon}+\frac\varepsilon{1+\varepsilon}\cos t$.
Now, for all $t\in(0,\pi)$ and $\varepsilon>0$ we have $S^\varepsilon(t)>0$. Since
$C^\varepsilon(\pi)=\frac{\overline{\rho _{2} }}{1+\varepsilon}+\frac\varepsilon{1+\varepsilon}$ then  $C^\varepsilon(\pi)<\overline{\rho _{1} }$ and
$C^\varepsilon(\pi)\to\overline{\rho _{2} }$ as $\varepsilon\to 0+.$ Therefore $\overline{\rho _2}\le \overline{\rho _1},$ and thus
$\overline{\rho _2}= \overline{\rho _1}.$ 
These relations plus independence of the coefficients from $R$ means that
$$
J_1= \overline{\rho}= \overline{\rho _{1} }=\sup_{\sum a_j=1}\left\{\rho(a_1,\dots,a_n)\right\}=-\tan^2\frac\pi{2(n+1)}.
$$
The theorem is proved.
\end{proof}

Formula \eqref{8} implies that 
$$
\mu_n(1)=\cot^2\frac\pi{2(n+1)}.
$$

\textbf{\textit{Corollary}}. Let conjugate trigonometric polynomials \eqref{10}
be normalized by the condition $\sum _{j=1}^{n}a_{j} =1.$ And let $\mathcal T$ be the set of sign 
changes for the function $S(t)$ on  $\left(0,\, \pi \right)$. Then
\[\mathop{\max }\limits_{a_{1} ,\ldots ,a_{n} } \mathop{\min }\limits_{} \left\{\, \, C(t):\, t\in\mathcal T\cup \left\{\pi \right\}\, \right\}=-\tan^{2} \frac{\pi }{2(n+1)} ,\] 
and the coefficients of the extremal pairs of the trigonometric polynomials are defined uniquely by the formulas
\begin{equation} \label{13}
a_{j}^{0} =2\cdot \tan\frac{\pi }{2(n+1)} \cdot (1-\frac{j}{n+1} )\cdot \sin \frac{\pi j}{n+1} ,\; j=1,\, \ldots \, ,n.               
\end{equation} 

To prove the corollary it is enough to find the coefficients $a_{1}^{0} ,\, \ldots \, ,\, a_{n}^{0} $. The polynomial $\dfrac{S^{0} (t)}{\sin t} $ is proportional to the Fej\'er polynomial  $\Phi_n^{(1)}(t)$:
$$
\frac{S^{0} (t)}{\sin t} =\frac{1}{1+\cos \frac{\pi }{n+1} } +\frac{2\cos \frac{\pi }{n+1} }{1+\cos \frac{\pi }{n+1} } \cos t +...=
$$
$$
\frac{1-\cos \frac{\pi }{n+1} }{n+1} \cdot \frac{2\cos ^{2} \frac{n+1}{2} t}{(\cos t-\cos \frac{\pi }{n+1} )^{2} } =\gamma _{1}^{0} +2\gamma _{2}^{0} \cos t+...\, \, .
$$
From here the coefficients $\gamma _{1}^{0} ,\, \, \ldots \, \, ,\gamma _{n}^{0} ,\; a_{1}^{0} ,\, \ldots \, ,\; a_{n}^{0} $ are determined by the formulas
\[\gamma _{j}^{0} =\frac{(n-j+3)\sin \frac{\pi {\kern 1pt} j}{n+1} -(n-j+1)\sin \frac{\pi (j-2)}{n+1}}{2(n+1)\sin \frac{\pi }{n+1} \cdot (1+\cos \frac{\pi }{n+1} )},\] 
\[a_{j}^{0} =\gamma _{j}^{0} -\gamma _{j+2}^{0} =2\cdot \tan\frac{\pi }{2(n+1)} \cdot (1-\frac{j}{n+1} )\cdot \sin \frac{\pi \, j}{n+1} ,\; j=1,\, \, \ldots \, \, ,n.\] 
It  is assumed that $\gamma _{n+1}^{0} =\gamma _{n+2}^{0} =0$.\\

\textbf{\textit{Example}} For $n=5$ by the formula  \eqref{13}
we get $a_{1}^{0} =\frac{5}{6} \cdot \tan\frac{\pi }{12} $, $a_{2}^{0} =\frac{2\sqrt{3} }{3} \cdot \tan\frac{\pi }{12} $, $a_{3}^{0} =\tan\frac{\pi }{12} $, $a_{4}^{0} =\frac{\sqrt{3} }{3} \cdot \tan\frac{\pi }{12} $, $a_{5}^{0} =\frac{1}{6} \cdot \tan\frac{\pi }{12} $. The graph and a fragment of the image of the unit circle is displayed on the Figures 1 and 2.

\begin{figure}[ht]
%\centering
\begin{minipage}[b]{0.45\linewidth}
\includegraphics[scale=0.15]{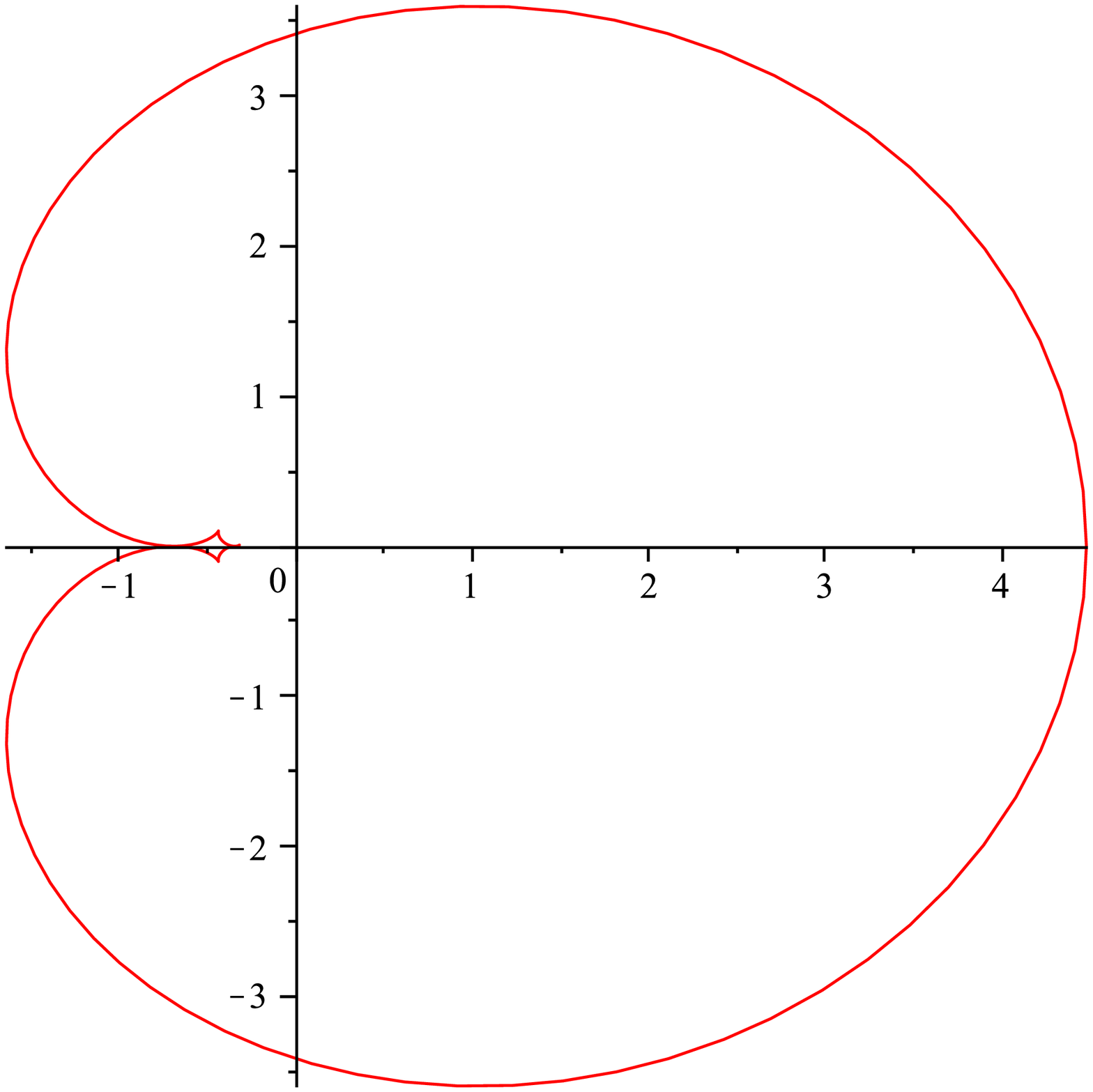}
\caption{Graph of the curve $x(t)=\Re\left(\sum _{j=1}^{n}a_{j}^{0} e^{-ijt}  \right),\, \, y(t)=\Im\left(\sum _{j=1}^{n}a_{j}^{0} e^{-ijt}  \right)$ for $n=5$ $t\in \left[0,\, 2\pi \right]$.}
\label{Fig1}
\end{minipage}
\;
\begin{minipage}[b]{0.45\linewidth}
\includegraphics[scale=0.15]{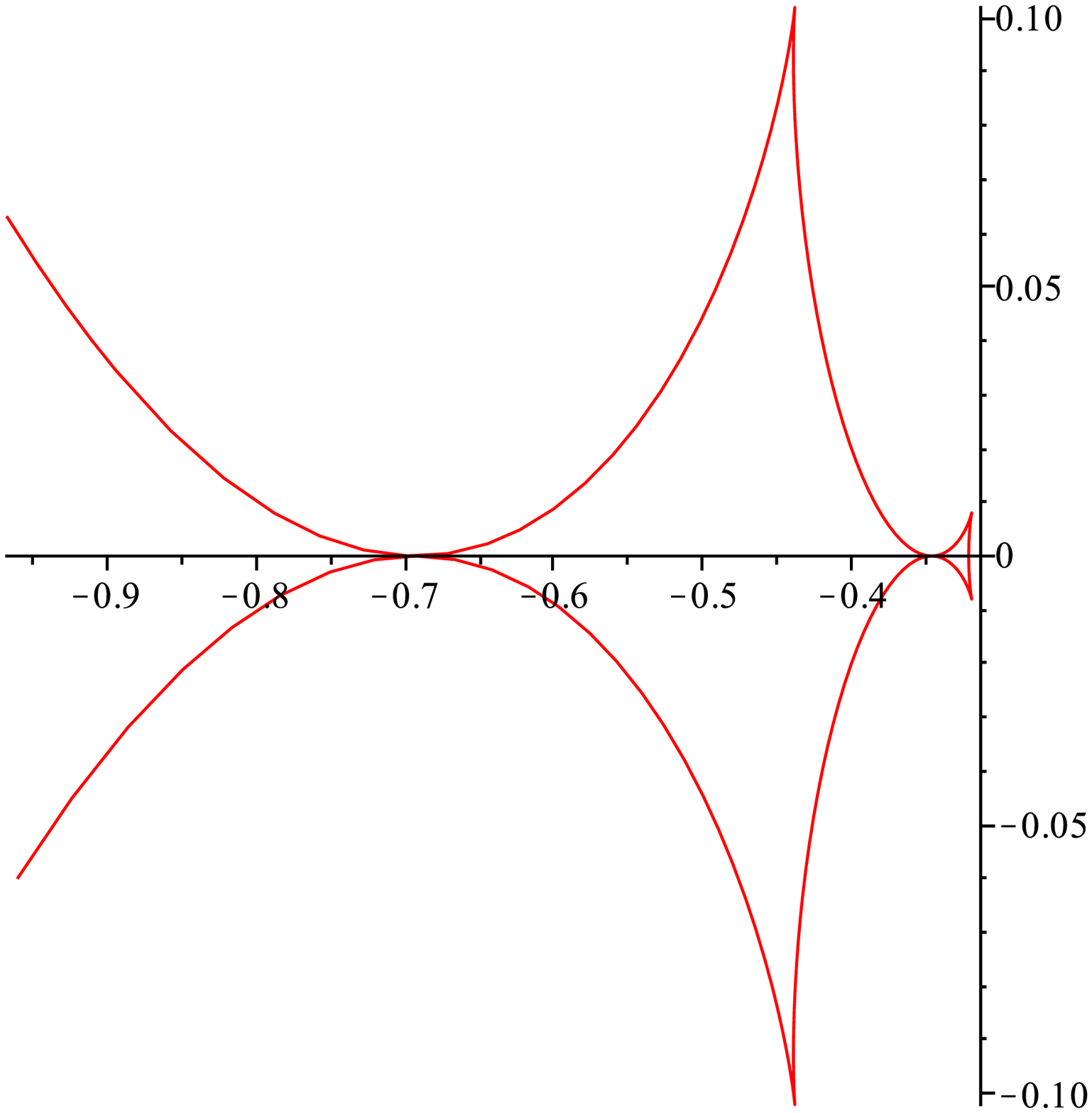}
\caption{ Graph of the curve $x(t)=\Re\left(\sum _{j=1}^{n}a_{j}^{0} e^{-ijt}  \right),\, \, y(t)=\Im\left(\sum _{j=1}^{n}a_{j}^{0} e^{-ijt}  \right)$ for $n=5,$ $t\in \left[1.43,\, 4.85\right]$. }
\label{Fig2}
\end{minipage}
\end{figure}

% Fig. \ref{Fig1}: 
% Fig. \ref{Fig2}:

Theorem 2 defines the optimal polynomial mapping for T=1:
\begin{equation} \label{14}
z\cdot p_{0}^{(1)} \left(z\right)=2\cdot \tan\frac{\pi }{2(n+1)} \cdot \sum _{j=1}^{n}(1-\frac{j}{n+1} )\cdot \sin \frac{\pi \, j}{n+1} \cdot z^{j},                      
\end{equation} 
and 
$$
\Im\left\{e^{it} \cdot p_{0}^{(1)} \left(e^{it} \right)\right\}=\frac{2\left(1-\cos \frac{\pi }{n+1} \right)}{n+1} \cdot \sin t\cdot \Phi _{n-1}^{(1)} \left(t\right).
$$
Note that the polynomials \eqref{14} are not new. Suffridge \cite{Su} used polynomials
\begin{equation} \label{15}
s_{k,n} (z)=z+\frac{1}{\sin \frac{k\pi }{n+1} } \cdot \sum _{j=2}^{n}\frac{n-j+1}{n} \cdot \sin \frac{k\pi \, j}{n+1} \cdot z^{j}   
\end{equation} 
to obtain sufficient conditions of the univalency for the polynomials with the first coefficient 1 and with the  leading  coefficient $a_n=1/n.$  D.Dimitrov \cite{D1}  established a relation between
the polynomials  \eqref{15} for $k=1$ and Fej\'er kernels $\Phi _{n}^{(1)} \left(t\right).$ From there for the univalent polynomials of the type  $q(z)=z+\, \ldots $ he derived the sharp estimates for the ratio $q(-1)/q(1).$ Let us note that Dimitrov's estimate can be obtained from Theorem 2, however the Theorem 2 does not follow from the results of \cite{D1} because the initially considered mappings are not necessary univalent. Moreover, as it is proven in Lemma 3, a necessary condition of optimality - the image of the upper semi-disc lies in the upper half plain - does not implies the univalency. One can judge about the univalency only at the end of the proof when the optimal polynomial is constructed.

The univalency of the polynomials \eqref{15} was established in \cite{Su} and since 
$$
zp^{(1)}_0(z)=4\frac n{n+1}\cdot\sin^2\frac\pi{2(n+1)}\cdot s_{1,n}(z)
$$ 
then the polynomial \eqref{14} is univalent.

\subsection{T=2 case}

\begin{theorem}\label{trm3}
Let $C(t)$ and $S(t)$ be a pair of conjugate trigonometric polynomials
\begin{equation} \label{16}
C(t)=\sum _{j=1}^{n}a_{j} \cos (2j-1)t,\quad  S(t)=\sum _{j=1}^{n}a_{j} \sin (2j-1)t,  
\end{equation}
%$$
%C(t) = \sum\limits_{j = 1}^n {{a_j}\cos (2j - 1)t,\quad } S(t) = \sum\limits_{j = 1}^n {{a_j}\sin (2j - 1)t,} 
%$$
normalized by the condition $\sum\limits_{j = 1}^n {{a_j} = 1}.$
Consider the extremal problem
$$
J_2 = \mathop {\inf }\limits_{\sum\limits_{j = 1}^n {{a_j}}  = 1} 
\left[ {\mathop {\max }\limits_{t \in \left[ {0,\frac{\pi }{2}} \right]} \left\{ {|S(t)|:C(t) = 0} 
\right\}} \right].
$$
Then 
$$  
J_2 = \frac1n.
$$
\end{theorem}

The function  $C(t)$ turns to be zero for $t=\frac{\pi }{2} $ and any $a_{1} ,\, \ldots \, ,\, a_{n} $. Denote 
$$
\rho (a_{1} ,\, \ldots \, ,a_{n} )=\mathop{\max }\limits_{t\in \left[\, 0,\, \frac{\pi }{2} \, \right]} \left\{S(t):C(t)=0\, ,\, S(t)>0\right\}.
$$ 
It is clear that 
$$
J_2\ge \mathop{\inf }\limits_{(a_{1} ,\, \ldots \, ,\, a_{n} )} \, \left\{\, \rho (a_{1} ,\, \ldots \, ,a_{n} )\, \right\}.
$$ 
The quantity  $\inf \, \left\{\, \rho (a_{1} ,\, \ldots \, ,a_{n} )\, \right\}$ will be minimized on the set  
$$A_{R} =\left\{\left(a_{1} ,\, \ldots \, ,a_{n} \right):\sum _{j=1}^{n}a_{j} =1,\; \sum _{j=1}^{n}\left|a_{j} \right|\le R  \right\},$$ where  $R$ is a sufficiently large number.

For $t\in \left(0,\, \frac{\pi }{2} \right)$ the following inequalities
\[\begin{array}{l} {C(t)=\frac{1}{2\sin t} \sum _{j=1}^{n}\hat{a}_{j} \sin 2jt\, \,  ,\, \, \, \, } \\ {S(t)=\frac{1}{2\sin t} \left(\sum _{j=1}^{n}\hat{a}_{j} -\sum _{j=1}^{n}\hat{a}_{j} \cos 2jt  \right)\, ,} \end{array}\] 
 are valid, where $\hat{a}_{j} =a_{j} -a_{j+1} ,\, \, j=1,\, \ldots \, ,\, n$ (consider $a_{n+1} =0$). It is easy to see that  $\sum _{j=1}^{n}\hat{a}_{j} =a_{1}.$ \\

Since trigonometric polynomials  $2\sin t\cdot \, C(t)$ and $\sum _{j=1}^{n}\hat{a}_{j} -2\sin t\, \cdot S(t) $  are conjugate we can apply Theorem 1 and prove the following lemma

\begin{lemma}
Let $C(t_{1} )=\, 0,\; \; S(t_{1} )=\, \, \gamma \, $, where $t_{1} \in \left(0,\, \frac{\pi }{2} \right)$. Then trigonometric polynomials \eqref{16}
admit the presentation
\[C(t)=\frac{1}{2\sin t} \cdot (\cos 2t-\cos 2t_{1} )\sum _{j=1}^{n-1}\alpha _{j}  \sin 2jt,\] 
\[S(t)=\frac{1}{2\sin t} \left(a_{1} +\frac{\alpha _{1} }{2} -(\cos 2t-\cos 2t_{1} )\sum _{j=1}^{n-1}\alpha _{j}  \cos 2jt\right),\] 
where  $\alpha _{j} ,\, j=1,\, \ldots \, ,\, n-1$ are uniquely determined by the coefficients $a_{1} ,\, \ldots \, ,a_{n} $ and parameter $\, \gamma \, $, moreover  $\dfrac{\alpha _{1} }{2} =2\gamma \sin t_{1} -a_{1} $. 
\end{lemma}

Together with the function $\rho (a_{1} ,\, \ldots \, ,a_{n} )$ we will consider the function
\[\rho _{1} (a_{1} ,\, \ldots \, ,a_{n} )=\mathop{\max }\limits_{t\in \left[\, 0,\, \frac{\pi }{2} \, \right]} \left\{S(t):t\in\mathcal T\cup \left\{\frac{\pi }{2} \right\}\, \right\},\] 
where $\mathcal T$ is the set of the points from the interval $(\, 0,\, \frac{\pi }{2} \, )$ where the function $C(t)$ changes the sign and the function $S(t)$ is positive.

\begin{lemma}\label{lm5} There exists a pair of trigonometric polynomials  $\left\{C^{0} (t),S^{0} (t)\right\}$, such that
\[
\bar\rho_1:=\mathop{\inf }\limits_{(a_{1} ,\, \ldots \, ,\, a_{n} )} \left\{\rho _{1} (a_{1} ,\, \ldots \, ,a_{n} )\right\}=\mathop{\max }\limits_{t\in \left[\, 0,\, \frac{\pi }{2} \, \right]} \left\{S^{0} (t) : t \in\mathcal T^{0} \cup \left\{\frac{\pi }{2} \right\}\, \right\},
\] 
where $\mathcal T^{0} $ is an intersection of the sets of  sign changes  of function  $C^{0} (t)$ over the interval $(0,\frac{\pi }{2})$, and the set of positivity of the function $S^{0} (t)$.  
\end{lemma}

\begin{proof}  Let 
$$
A_{R} =\left\{\left(a_{1} ,\, \ldots \, ,a_{n} \right):\sum _{j=1}^{n}a_{j} =1,\; \sum _{j=1}^{n}\left|a_{j} \right|\le R  \right\}.
$$ 
The function $\rho _{1} (a_{1} ,\, \ldots \, ,a_{n} )$ is continuous on the set  $A_{R} $, beside  $\left(\, a_{1} ,\, \, \ldots \, \, ,a_{n} \right)$, for which the maximal value of  $S(t)$ is achieved at zeros of the function  $C(t)$ where it does not change the sign. The lower limit of the function  $\rho _{1} (a_{1} ,\, \ldots \, ,a_{n} )$ at the point of discontinuity is equal to the value of the function. So, it does achieve a minimal value
$\overline{\rho _{1} }$ on the set  $A_{R} $, i.e. $\overline{\rho _{1} }=\mathop{\min }\limits_{(a_{1} ,\, \ldots \, ,a_{n} )\in A_{R} } \left\{\rho _{1} (a_{1} ,\, \ldots \, ,a_{n} ) \right\}$.

It is clear that  $0\le \overline{\rho _{1} }\le 1$ and by Lemma 1 for $R>2^{2n-1} $ the maximum is achieved in the internal point of the set  $A_{R}. $ The optimal pair 
 $\left\{C^{0} (t),S^{0} (t)\right\}$, where this maximum is achieved, is independent of  $R$.  Lemma \ref{lm5} is proved. 
 \end{proof}

Since $\mathcal T\cup \left\{\frac{\pi }{2} \right\}$ is a subset of all zeros of the functions $C(t)$, then  $\overline{\rho }\ge \overline{\rho _{1} }$, where $\overline{\rho }=\mathop{\inf }\limits_{(a_{1} ,\, \ldots \, .,a_{n} )\in A_{R} } \left\{\, \, \rho (a_{1} ,\, \ldots \, ,a_{n} )\, \, \right\}$. 

\begin{lemma}\label{pos1}
If a polynomial $C(t)$ has zero in $(0, \pi/2),$ then it cannot be optimal.
\end{lemma}

\begin{proof}
Suppose that for the optimal polynomial ${C^0}(t)$ the set $\mathcal T=\{t_1,\dots,t_q\},$ where
$0 \le q \le n - 1$ , is not empty. Assume that  
$$
\max \left\{ {\,\,{S^0}({t_1}),...,{S^0}({t_q})\,}\right\} = {S^0}({t_1}),
$$ 
$$
{S^0}({t_1}) = {S^0}({t_j}),j = 1,...,m\;(1 \le m \le q),
$$
and 
$$
{S^0}({t_1}) > {S^0}({t_j}),\,\,j = m + 1,\,\, \ldots \,,\,q.
$$
The following three cases are possible: 
$$
{S^0}({t_1}) > {S^0}(\frac{\pi }{2}),\,  {S^0}(\frac{\pi }{2})\ge{S^0}({t_1})\ge 0,  {S^0}({t_1})<0.
$$

{\it Case 1.} Accordingly to Theorem 1 %\ref{lem6}, \ref{lem7}
 the trigonometric polynomials ${S^0}(t)$ and ${C^0}(t)$ have the form
$$
{C^0}(t) = \frac{1}{{2\sin t}} {(\cos 2t - \cos 2{t_1})}   \sum\limits_{j = 1}^{n - 1} {\alpha_j} \sin 2jt,
$$
$$
{S^0}(t) = \frac{1}{{2\sin t}}\left( {\alpha _1  + \frac{\alpha_1}{2} -  {(\cos 2t - \cos 2{t_1})}  \sum\limits_{j = 1}^{n - 1} 
{\alpha_j} \cos 2jt} \right).
$$
Since ${S^0}(t_1)=\frac1{2\sin t_1}(\alpha_1+\frac{\alpha_1}2),$ then $\alpha_1+\frac{\alpha_1}2>0.$ And since $C^0(0)=1$ then
$$
(1-\cos 2t_1)\sum_{j=1}^{n-1}j\alpha_j=1,\qquad\mbox{therefore}\qquad \sum_{j=1}^{n-1}j\alpha_j>0.
$$

Let us construct the auxiliary trigonometric polynomials ${S^0}(t)$ and ${C^0}(t)$ 
$$
C(\theta,t)= N(\theta)\frac{1}{{2\sin t}}{(\cos 2t - \cos 2\theta)}  \sum\limits_{j = 1}^{n - 1} {\alpha_j} \sin 2jt,
$$
$$
S(\theta,t) = N(\theta) \frac{1}{{2\sin t}}\left( {\alpha _1  + \frac{\alpha_1}{2} -  {(\cos 2t - \cos 2\theta)}   \sum\limits_{j = 1}^{n - 1} 
{\alpha_j} \cos 2jt} \right),
$$
where the normalization factor  $N(\theta)$ provides the sums of the coefficients for each of the polynomials $C(\theta,t)$ and $S(\theta,t)$ 
to be 1. For the polynomial $C(\theta,t)$ the set of sign changes  is $\mathcal T_{\theta } =\left\{\theta ,\, \, t_{2} ,\, \ldots \, ,\, t_{q} \right\}$. It is clear that $C(t_{1} ;\, t)\equiv C^{0} (t)$, $S(t_{1} ;\, t)\equiv S^{0} (t)$. 
The normalizing factor $N(\theta)$ is determined by the condition $C(\theta,0)=1,$ i.e.
$$
N(\theta)=\frac1{(1-\cos2\theta)\sum_{j=1}^{n-1}j\alpha_j}. 
$$
The polynomials $C(\theta ;\, t)$ and $S(\theta ;\, t)$ can be written in the following form
\[C(\theta ;\, t)=\frac{1}{(1-\cos 2\theta )\sum _{j=1}^{n-1}j\alpha _{j}  } \cdot \frac{1}{2\sin t} (\cos 2t-\cos 2\theta )\sum _{j=1}^{n-1}\alpha _{j} \sin 2jt ,\] 
\[S(\theta ;\, t)=\frac{a_{1} +\frac{\alpha _{1} }{2} -(\cos 2t-\cos 2\theta )\sum _{j=1}^{n-1}\alpha _{j}  \cos 2jt}{\left((1-\cos 2\theta )\sum _{j=1}^{n-1}j\alpha _{j}\right) 2\sin t}.\] 

Let us show that for some $\theta$ the value $\rho _{1}(a_1,\dots,a_n) $ for the pair $\left\{C(\theta ;\, t),\, S(\theta ;\, t)\right\}$ is less then for the pair  $\left\{\, C^{0} (t),\, \, S^{0} (\, t)\right\}$, i.e. the pair  
 $\left\{\, C^{0} (t),\, \, S^{0} (\, t)\right\}$ is not an optimal. Let compute
\[S(\theta ;\, \theta )=\frac{1}{(1-\cos 2\theta )\sum _{j=1}^{n-1}j\alpha _{j}  } \cdot \frac{1}{2\sin \theta } \left(a_{1} +\frac{\alpha _{1} }{2} \right)\] 
\[S(\theta ;\, t_{k} )=\frac{a_{1} +\frac{\alpha _{1} }{2} -(\cos 2t_{k} -\cos 2\theta )\sum _{j=1}^{n-1}\alpha _{j}  \cos 2jt_{k} }{\left((1-\cos 2\theta )\sum _{j=1}^{n-1}j\alpha _{j} \right) 2\sin t_{k} }, k=2,\, \ldots \, ,\, m.\] 
Since $S^{0} (t_{1} )=S^{0} (t_{k} ),k=2,...,m\; $, then
$$
\sum _{j=1}^{n-1}\alpha _{j}  \cos 2jt_{k} =\frac{a_{1} +\frac{\alpha _{1} }{2} }{\sin t_{1} }\cdot \frac{\sin t_{1} -\sin t_{k} }{\cos 2t_{k} -\cos 2t_{1} } >0.
$$ 
Hence, the functions $S(\theta ;\, \theta )$, $S(\theta ;\, t_{k} )$, $k=2,...,m$ are decreasing with respect to the parameter $\theta $. 

Now, let $0<\theta -t_{1} <\varepsilon$. Because of the continuity of trigonometric polynomials by $t$ and by all coefficients we have the following inequalities valid
 $$
 S(\theta ;\, \theta )<S^{0} (t_{1} ),
 $$
 $$ 
 S(\theta ;\, t_{j} )<S^{0} (t_{j} ),\, j=2,\, \ldots \, ,\, m\,,
 $$
  $$
 \left|S(\theta ;\, t_{j} )-S^{0} (t_{j} )\right|<\delta ,\, j=m+1,\, \, \ldots \, \, ,\, q,
 $$ 
$$
\left|S(\theta ;\, \frac{\pi }{2} )-S^{0} (\frac{\pi }{2} )\right|<\delta 
$$ 
for any  $\delta $ with a proper choice of $\varepsilon$. These inequalities mean that the quantity 
$$
\max \left\{S(\theta ;\, \theta ),\, \, S(\theta ;\, t_{2} ),\, \, \, \ldots \, ,\, S(\theta ;\, t_{q} )\, ,\, \, S(\theta ;\, \frac{\pi }{2} )\right\}
$$
is less then 
$$
\max \left\{\, \, S^{0} (t_{1} ),...,S^{0} (t_{q} ),S^{0} (\frac{\pi }{2} )\, \right\}
$$ 
at least for sufficiently small positive  $\theta -t_{1} $, i.e. the pair $\left\{C^{0} (t),\, \, S^{0} (\, t)\right\}$ is not an optimal.\\

{\it Case 2.} Let evaluate
\[S^{0} (\frac{\pi }{2} )=\frac{1}{2} \left(a_{1} +\frac{\alpha _{1} }{2} +(1+\cos 2t_{1} )\sum _{j=1}^{n-1}(-1)^{j} \alpha _{j} \, \,  \right),\] 
\[
S(\theta ;\, \frac{\pi }{2} )=\frac{a_{1} +\frac{\alpha _{1} }{2} +(1+\cos 2\theta )
\sum _{j=1}^{n-1}(-1)^{j} \alpha _{j}  }{2(1-\cos 2\theta )\sum _{j=1}^{n-1}j\alpha _{j}},
\] 
and  $S(t_{1} ;\, \frac{\pi }{2} )=S^{0} (\frac{\pi }{2} )$. Since by the assimption $S^{0} (\frac{\pi }{2} )\ge S^{0} (t_{1} )=\frac{1}{2\sin t_{1} } \left(a_{1} +\frac{\alpha _{1} }{2} \right)\ge 0$, 
then  $\sum _{j=1}^{n-1}(-1)^{j} \cdot \alpha _{j}  \ge 0$. Then the  functions $S(\theta ;\, \frac{\pi }{2} )$, $S(\theta ;\, \theta )$, $S(\theta ;\, t_{k} )$, $k=2,...,m$, are decreasing by the parameter $\theta.$ Therefore, in this case the pair  $\left\{C^{0} (t),\, S^{0} (\, t)\right\}$ cannot be an optimal.\\

{\it Case 3.} If $S^{0} (\, t_{1} )<0$, then $a_{1} +\frac{\alpha _{1} }{2} <0$. But $S^{0} (\frac{\pi }{2} )>0$, therefore $\sum _{j=m}^{n-1}(-1)^{j} \cdot \alpha _{j}  \ge 0$. Thus, in this case the function $S(\theta ;\, \frac{\pi }{2} )$ is decreasing by $\theta $ too, and the pair $\left\{C^{0} (t),\, S^{0} (\, t)\right\}$ cannot be an optimal. The lemma is proved and we can proceed with the proof of Theorem 3.
\end{proof}

So, $\mathcal T^{0} =\emptyset $ and therefore 
$$
\bar \rho_1=\min_{a_1+...+a_n=1}\max_{t\in[0,\pi/2]}\left\{|S(t)|: C(t)>0, t\in\left(0,\frac\pi2\right), C(t)=0\right\}=
$$
$$
\min_{a_1+...+a_n=1}\left\{\left|S\left(\frac\pi2\right)\right| : C(t)>0, t\in\left(0,\frac\pi2\right)\right\}
$$

The polynomial $C(t)$ can be written as 
\[C(t)=\cos t\cdot (\gamma _{1} +2\gamma _{2} \cos 2t+..+2\gamma _{n} \cos 2(n-1)t),\]
where the coefficients $a_{1} ,\, \, \ldots \, \, ,\, \, a_{n} $ and $\gamma _{1} ,\, \, \ldots \, ,\, \, \gamma _{n} $ are connected by  a bijective relations: 
$$
\gamma _{s} =\sum _{j=s}^{n}(-1)^{s+j} a_{j}  \; ,s=1,\, \, \ldots \, ,\, \, n.
$$ 
It is clear that  $\gamma _{1} +2\sum _{j=2}^{n}\gamma _{j}  =\sum _{j=1}^{n}a_{j} =1 $ and  
\[S(\frac{\pi }{2} )=\sum _{j=1}^{n}(-1)^{j+1} \gamma _{j}  =\gamma _{1} .\] 
Let
$$
\rho _{1} =\mathop{\min }\limits_{a_{1} +\ldots +a_{n}=1} 
\left\{\left|S\left(\frac\pi2\right)\right| : C(t)\ge 0,\, \, t\in (0,\, \, \frac{\pi }{2} )\, \right\}.
$$
Then
\[\rho _{1} =\mathop{\min }\limits_{\gamma _{1} ,\ldots ,\gamma _{n} } \left\{\, |\gamma _{1} |: \gamma _{1} +2\sum _{j=1}^{n}\gamma _{j}  =1,\, \, C(t)\ge 0,\, \, t\in (0,\, \, \frac{\pi }{2} )\, \right\}.\] 

From the properties of Fej\'er kernels ([2], 6.7, problem 50), it follows that the value of a non-negative trigonometric polynomial of degree $n-1$ with the zero term equals 1 does not exceed $n$. 
Moreover, if it is an even polynomial then the extremal values are achieved only at the points  $2k\pi ,\, k\in\mathbb Z$, and polynomial coefficients are determined uniquely. Fej\'er condition implies that the inequality for the extremal polynomial $\dfrac{C^{0} (t)}{\gamma _{1} } \le n$. Therefore, $\gamma _{1}^{0} =\frac{C^{0} (0)}{n} =\frac{1}{n} $ and $\rho _{1} =\frac{1}{n} $.

Let us find $\bar\rho_1$.  To do that, consider  one-parameter family of trigonometric polynomials 
\[C^{\varepsilon } (t)=\sum _{j=1}^{n}a_{j}^{\varepsilon } \cos (2j-1)t, S^{\varepsilon } (t)=\sum _{j=1}^{n}a_{j}^{\varepsilon } \sin (2j-1)t,\] 
where
\[a_{1}^{\varepsilon } =\frac{a_{1}^{0} +\varepsilon }{1+\varepsilon } ,\, \, a_{j}^{\varepsilon } =\frac{a_{j}^{0} }{1+\varepsilon } ,\, j=2,\, \, \ldots \, ,\, \, n.\] 
It is clear that 
\[\sum _{j=1}^{n}a_{j}^{\varepsilon } = \frac{a_{1}^{0} +\varepsilon }{1+\varepsilon } +\frac{a_{2}^{0} }{1+\varepsilon } +\, \ldots \, +\frac{a_{n}^{0} }{1+\varepsilon } =1,\] 
and
\[
C^{\varepsilon } (t)=\frac{C^{0} (t)}{1+\varepsilon } +\frac{\varepsilon }{1+\varepsilon } \cos t,\quad
S^{\varepsilon } (t)=\frac{S^{0} (t)}{1+\varepsilon } +\frac{\varepsilon }{1+\varepsilon } \sin t. 
\] 
Note that $C^{\varepsilon } (t)>0$ for all  $t\in (0,\frac{\pi }{2} )$ and $\varepsilon >0.$ Therefore,
\[\bar\rho_1 \le S^{\varepsilon } (\frac{\pi }{2} )=\frac{S^{0} (\frac{\pi }{2} )}{1+\varepsilon } +\frac{\varepsilon }{1+\varepsilon } .\]

If $\varepsilon \to 0+0$ in the limit then $\bar\rho_1 \le S^{0} (\frac{\pi }{2} )=\rho _{1} $. Since $\bar\rho_1 \ge \rho _{1} $, then $\bar\rho_1 =\rho _{1} =\frac{1}{n} $.

Let us find the coefficients  $a_{_{1} }^{0} ,\, \ldots \, ,\, a_{_{n} }^{0} $. Using Fej\'er kernel we obtain
$$
C^{0} (t)=\left(\frac{\sin nt}{n\sin t} \right)^{2} \cos t=\cos t\, \left(\frac{1}{n} +2\sum _{j=2}^{n}\frac{n-j+1}{n^{2} } \cos 2(j-1)t \right).
$$ 
Since $a_{j}^{0} =\gamma _{j}^{0} +\gamma _{j+1}^{0} $, $\; j=1,\, \, \ldots, n,$ assuming $\gamma _{n+1}^{0} =0$, then
\begin{equation} \label{17}
a_{j}^{0} =\frac{2(n-j)+1}{n^{2} } ,\; j=1,\, \ldots \, ,n. 
\end{equation}
Let compute
\[\mathop{\max }\limits_{t} \left\{\, \left|S^{\varepsilon } (t)\right|:C^{\varepsilon } (t)=0\, \right\}=\frac{1}{1+\varepsilon } \left(\frac{1}{n} +\varepsilon \right).\] 
From here $J_{2} =\dfrac{1}{n} $.\\

The proof of the theorem is completed.
\\

The formula \eqref{9} implies \[\mu _{n} (2)=n^{2} .\] 

\textbf{\textit{Corollary}} Let $C(t)$ and $S(t)$ be a pair of trigonometric polynomials  \eqref{16},
normalized by the condition  $\sum _{j=1}^{n}a_{j} =1 $.  Let  $T$  be a set of sign changes of the function  $C(t)$. Then 
$$
\mathop{\min }\limits_{(a_{1} ,\, \ldots \, ,\, a_{n} )} \mathop{\max }\limits_{t} \left\{\, \left|S(t)\right|:t\in\mathcal T\right\}=\frac{1}{n}.
$$ 
Moreover, the coefficients of the extremal pair of the trigonometric polynomials are defined in a unique way.

Accordingly to Theorem 3 the optimal polynomial mapping for $T=2$ is 
\[z\cdot \left(p_{0}^{(2)} \left(z\right)\right)^{2} =z\cdot \left(\sum _{j=1}^{n}\frac{2(n-j)+1}{n^{2} } \cdot z^{j-1}  \right)^{2} ,\] 
and the coefficients of the polynomial $p_{0}^{2} \left(z\right)$ are connected with the Fej\'er kernel $\Phi _{n}^{(2)} \left(t\right)$ by the relation
\[\Re\left\{e^{it} \cdot p_{0}^{(2)} \left(e^{2it} \right)\right\}=\frac{1}{n^{2} } \cdot \cos t\cdot \Phi _{n-1}^{(2)} \left(2t\right).\] 

\begin{figure}[ht]
%\centering
\begin{minipage}[b]{0.45\linewidth}
\hspace{2cm}\includegraphics[scale=0.15]{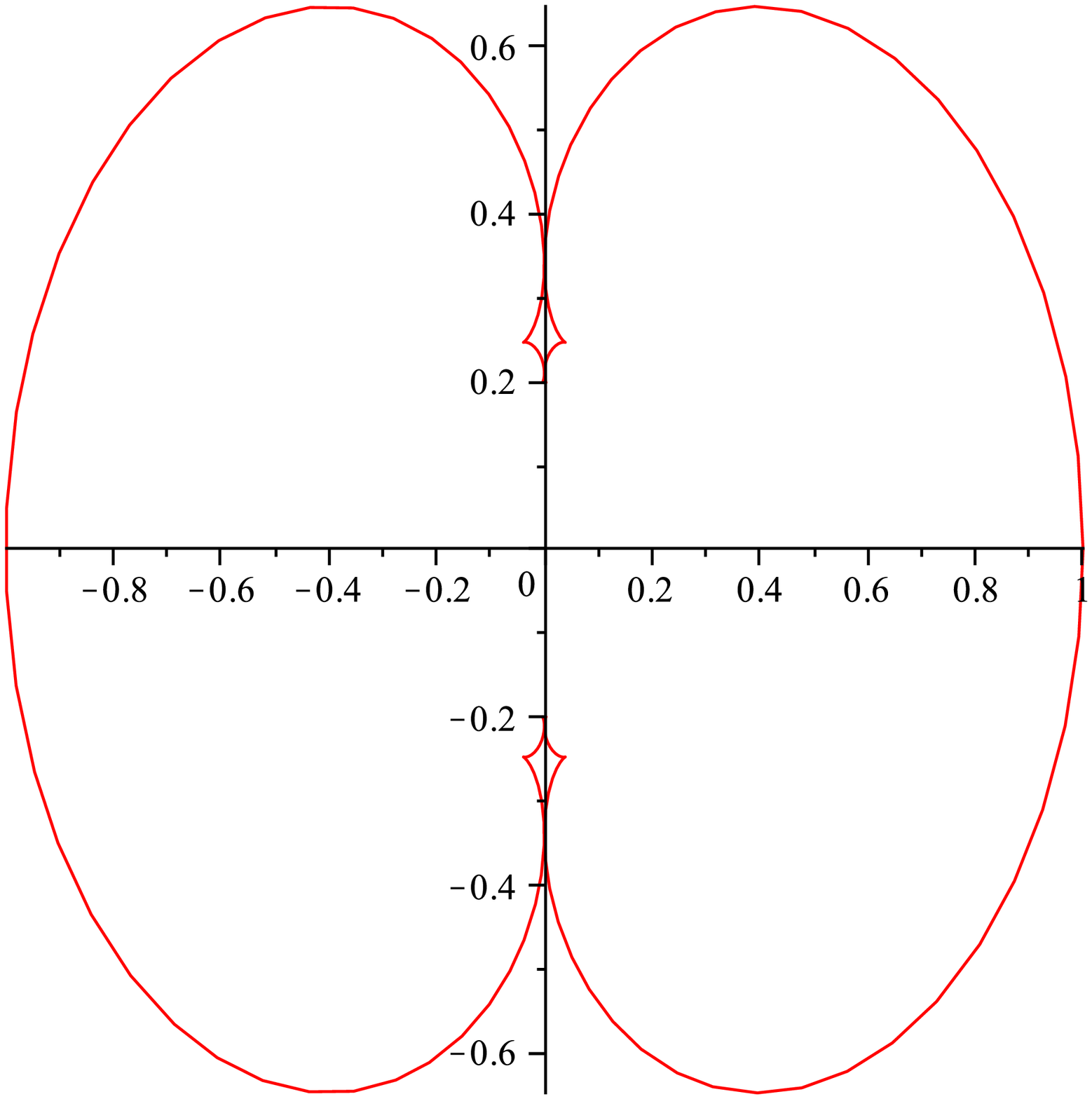}
\caption{The graph of the curve  $x(t)=\Re\left(\sum _{j=1}^{n}a_{j}^{0} e^{-i(2j-1)t}  \right),$ $\, y(t)=\Im\left(\sum _{j=1}^{n}a_{j}^{0} e^{-i(2j-1)t}  \right)$ for $n=5,$ $t\in \left[0,\, 2\pi \right]$.}
\label{Fig3}
\end{minipage}
\;
\begin{minipage}[b]{0.45\linewidth}
\hspace{2cm}\includegraphics[scale=0.15]{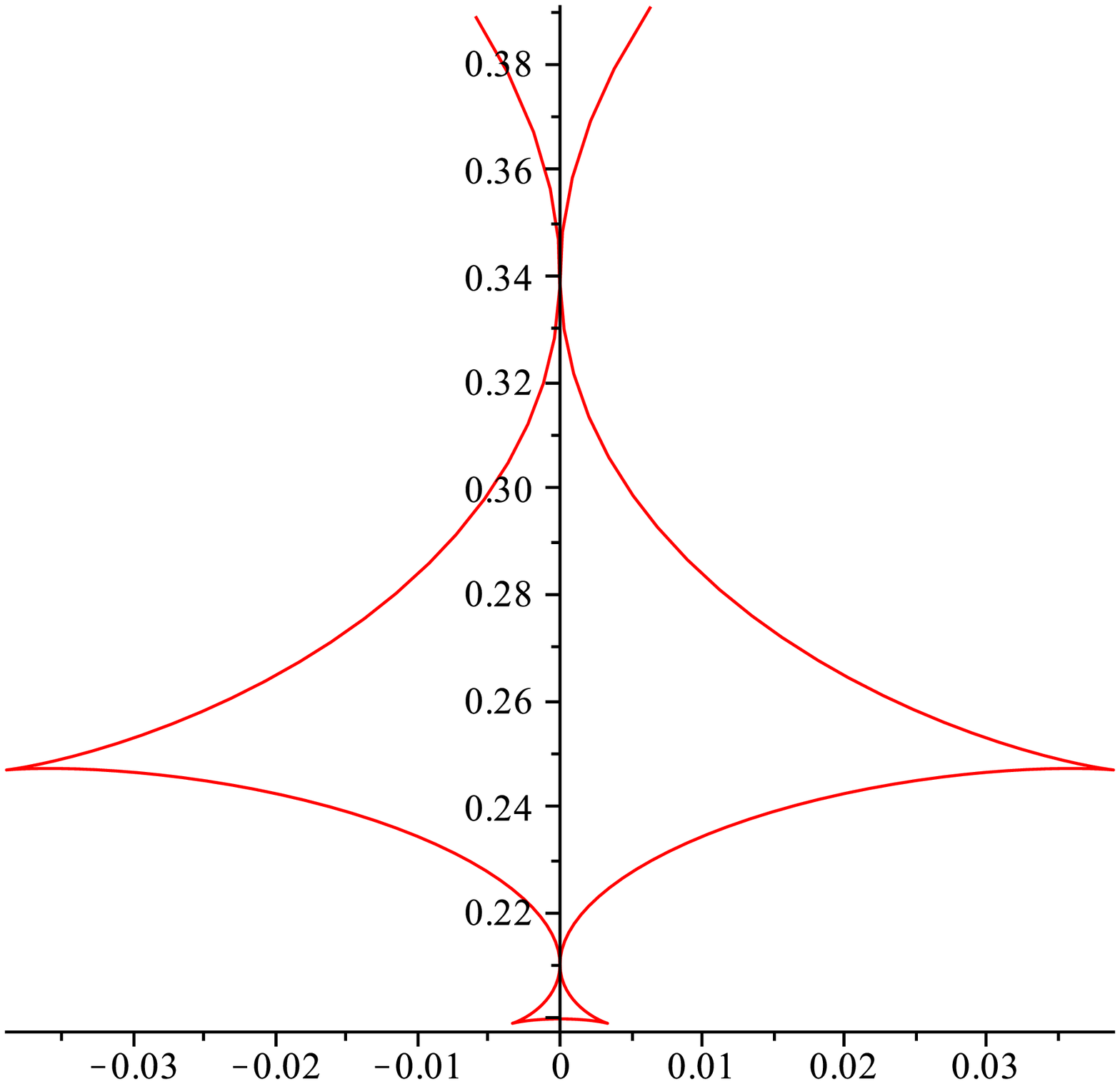}
\caption{The graph of the curve  $x(t)=\Re\left(\sum _{j=1}^{n}a_{j}^{0} e^{-i(2j-1)t}  \right),$ $ y(t)=\Im\left(\sum _{j=1}^{n}a_{j}^{0} e^{-i(2j-1)t}  \right)$ for $n=5,$ $t\in \left[0.58,\, 2.56\right]$. }
\label{Fig4}
\end{minipage}
\end{figure}

%Fig. \ref{Fig3}:  Fig. \ref{Fig4}:  

%\centerline{ \includegraphics*[width=1.83in, height=1.67in, keepaspectratio=false]{image9.eps}}
%\centerline{\includegraphics*[width=2.00in, height=2.00in, keepaspectratio=false]{image10.eps}}
%\centerline{ \includegraphics*[width=2.00in, height=2.00in, keepaspectratio=false]{image13.eps}}

\begin{figure}[ht]
%\centering
\begin{minipage}[b]{0.45\linewidth}
\hspace{2cm}\includegraphics[scale=0.15]{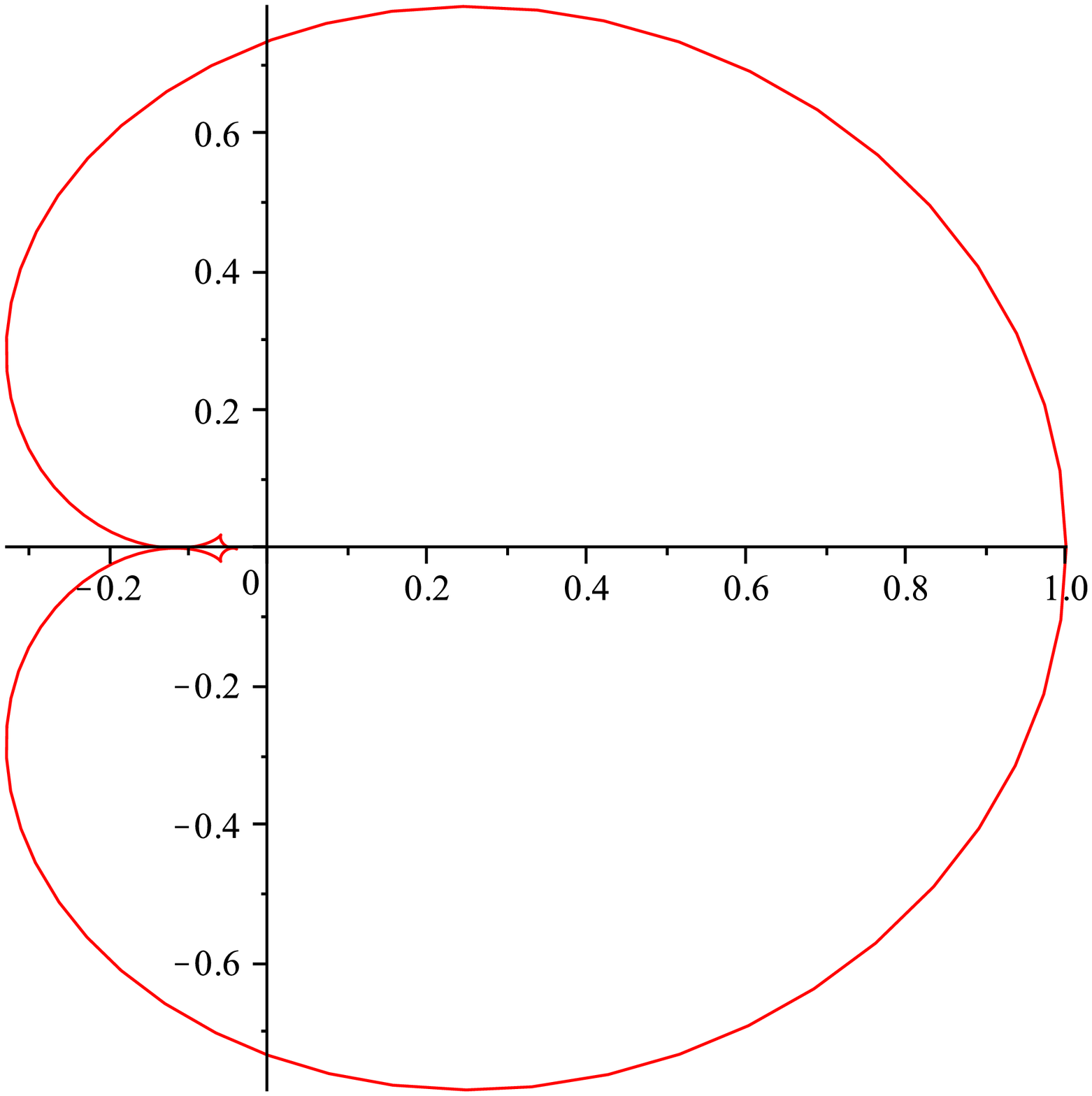}
\caption{ The graph of the curve  $x(t)=\Re\left(e^{-it} \left(\sum _{j=1}^{n}a_{j}^{0} e^{-i(j-1)t}  \right)^{2} \right),$
$y(t)=\Im\left(e^{-it} \left(\sum _{j=1}^{n}a_{j}^{0} e^{-i(j-1)t}  \right)^{2} \right)$ for $n=5,$ $t\in \left[0,\, 2\pi \right]$.}
\label{Fig5}
\end{minipage}
\;
\begin{minipage}[b]{0.45\linewidth}
\hspace{2cm}\includegraphics[scale=0.15]{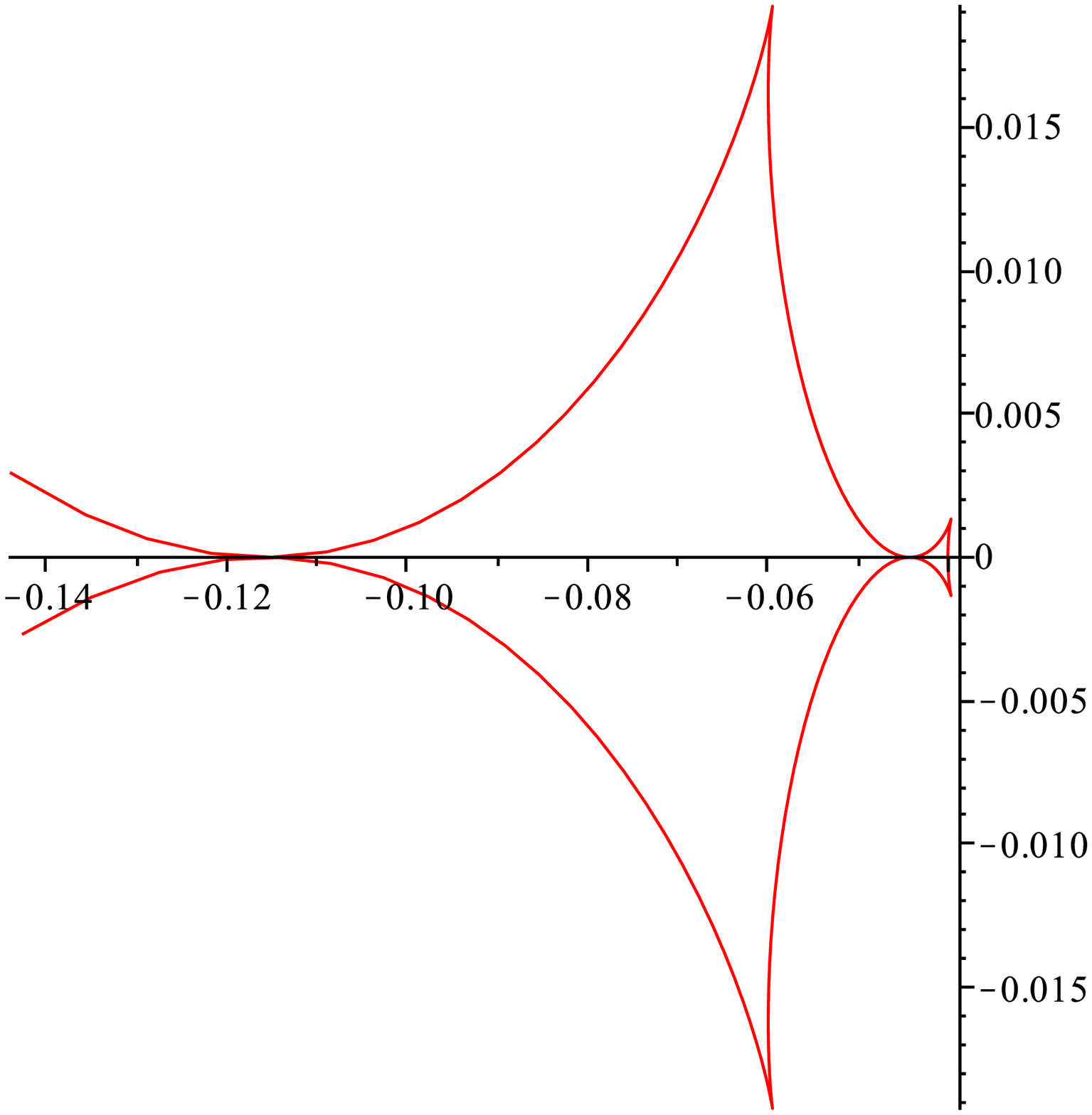}
\caption{The graph of the curve  $x(t)=\Re\left(e^{-it} \left(\sum _{j=1}^{n}a_{j}^{0} e^{-i(j-1)t}  \right)^{2} \right),$
$y(t)=\Im\left(e^{-it} \left(\sum _{j=1}^{n}a_{j}^{0} e^{-i(j-1)t}  \right)^{2} \right)$ for $n=5,$ $t\in \left[1.18,\, \, 5.1\right]$. }
\label{Fig6}
\end{minipage}
\end{figure}

%Fig. \ref{Fig5}:
%Fig. \ref{Fig6}: 

%\centerline{ \includegraphics*[width=2.00in, height=2.00in, keepaspectratio=false]{image14.eps}}

\begin{figure}[ht]
%\centering
%\begin{minipage}[b]{0.45\linewidth}
\hspace{2cm}\includegraphics[scale=0.15]{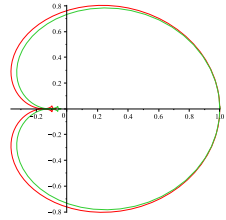}
\caption{ The graph from the Fig. \ref{Fig5}  (the outer) and from the Fig. \ref{Fig1}  (the inner).  It is 
quite remarkable that not only the values $\mu_n(1)$ and $\mu_n(2)$ are very close but
also  the whole optimal images.}
\label{Fig7}
%\end{minipage}
\end{figure}

%Fig. \ref{Fig7}:  \\

\begin{theorem}\label{trm4}
The mapping $zp_0^{(2)}(z^2)$ is univalent in  $\mathbb D.$
\end{theorem}

The proof follows from the univalency of the polynomial $s_{n,2n-1}(iz)$ \cite{Su} and the identity 
$$
zp_0^{(2)}(z^2)=-i\frac{n^2}{2n-1}\cdot s_{n,2n-1}(iz).
$$

{\bf Corollary} {\it  The mapping $z\left(p_0^{(2)}(z)\right)^2$ is univalent in $\mathbb D.$}

This follows from the univalency of $zp_0^{(2)}(z^2).$%  ( see \cite{PS}, 2.1 problems 79 and 80).--- ??????? May be Part III, 

\section{Optimal stabilization of chaos}

\subsection{ $T=1$.} For the one-parameter logistic mapping 
 \begin{equation}\label{18}
 f:\, \, \left[0,\, \, 1\right]\to \left[0,\, \, 1\right], 
\end{equation}
$f\left(x\right)=h\cdot x\cdot (1-x),\, \, \, 0\le h\le 4$
the equilibrium $x^*=1-\dfrac1h$ of the open-loop system 
$$
x_{k+1}=f(x_k),\; x_k\in\mathbb R,\; k=1,2,...
$$
is unstable for $h\in \left(3,\, \, 4\right].$ The multiplier $\mu \in \left[-2,\, -1\right)$. Then 
$$
\mu_1(1)=\cot^2\frac\pi4=1<\mu^*,\; \mu_2(1)=\cot^2\frac\pi6=3>\mu^*. 
$$
Therefore $n^*=1$ and the minimal depth of the prehistory in the delayed feedback is$N^*=1.$ The optimal strength
coefficient is $\epsilon_1^0=1/3$ and the required stabilizing control is 
$$
u=-\frac{1}{3} \left(f\left(x_{k} \right)-f\left(x_{k-1} \right)\right).
$$

%\centerline{ \includegraphics*[width=2.00in, height=2.00in, keepaspectratio=false]{image15.eps}}
%\centerline{ \includegraphics*[width=2.00in, height=2.00in, keepaspectratio=false]{image16.eps}}

\begin{figure}[ht]
%\centering
\begin{minipage}[b]{0.45\linewidth}
\hspace{2cm}\includegraphics[scale=0.15]{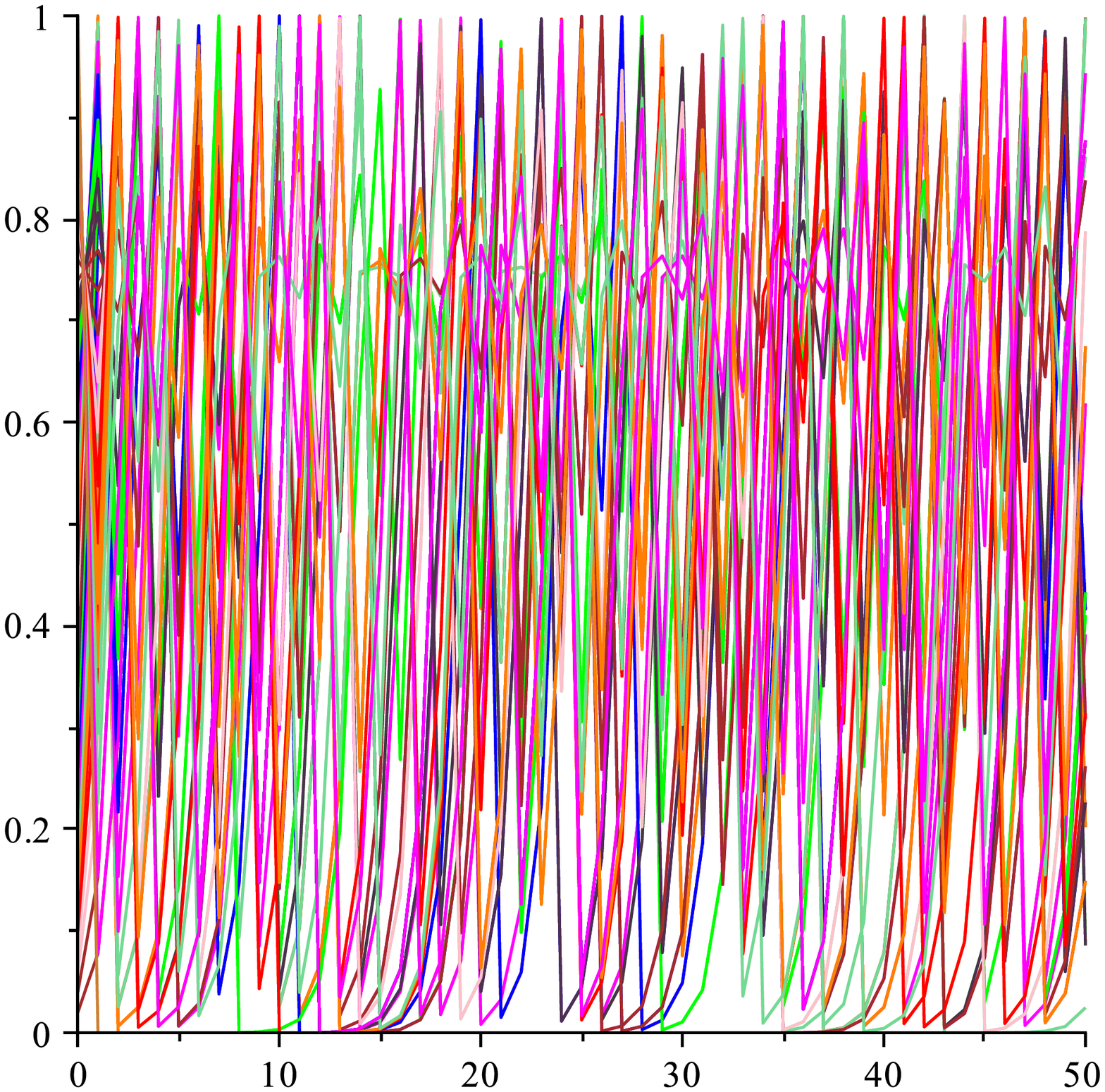}
\caption{Quasi-stochastic dynamics of the solutions to the logistic equations for $h=4.$  }
\label{Fig8}
\end{minipage}
\;
\begin{minipage}[b]{0.45\linewidth}
\hspace{2cm}\includegraphics[scale=0.15]{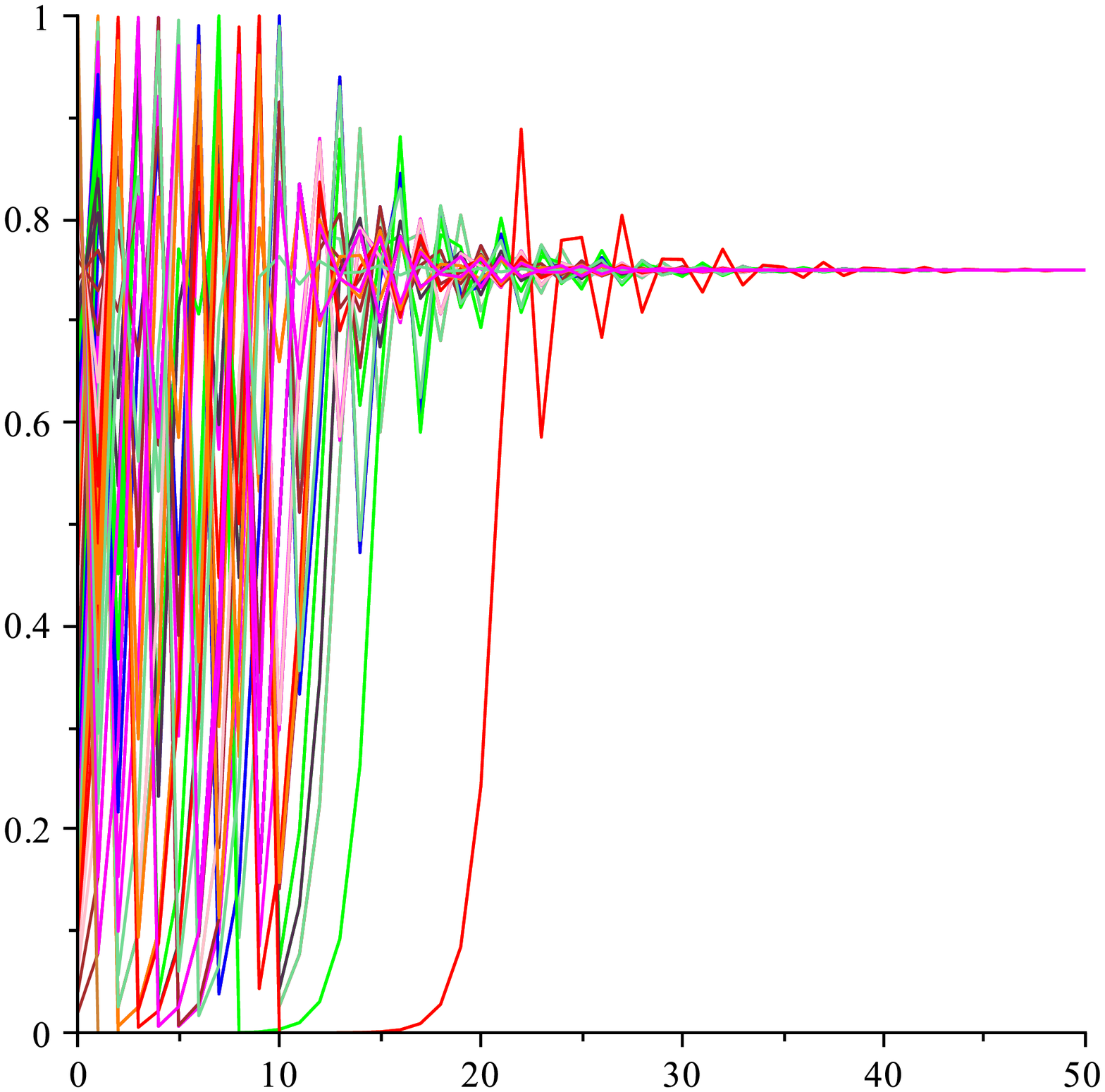}
\caption{ Dynamics of the solutions to the logistic equations for $h=4$  closed by stabilizing control 
$u=-\frac{1}{3} \left(f\left(x_{k} \right)-f\left(x_{k-1} \right)\right).$ }
\label{Fig9}
\end{minipage}
\end{figure}

%Fig. \ref{Fig8}: 
%Fig. \ref{Fig9}:

\begin{figure}[ht]
%\centering
%\begin{minipage}[b]{0.45\linewidth}
\hspace{2cm}\includegraphics[scale=0.15]{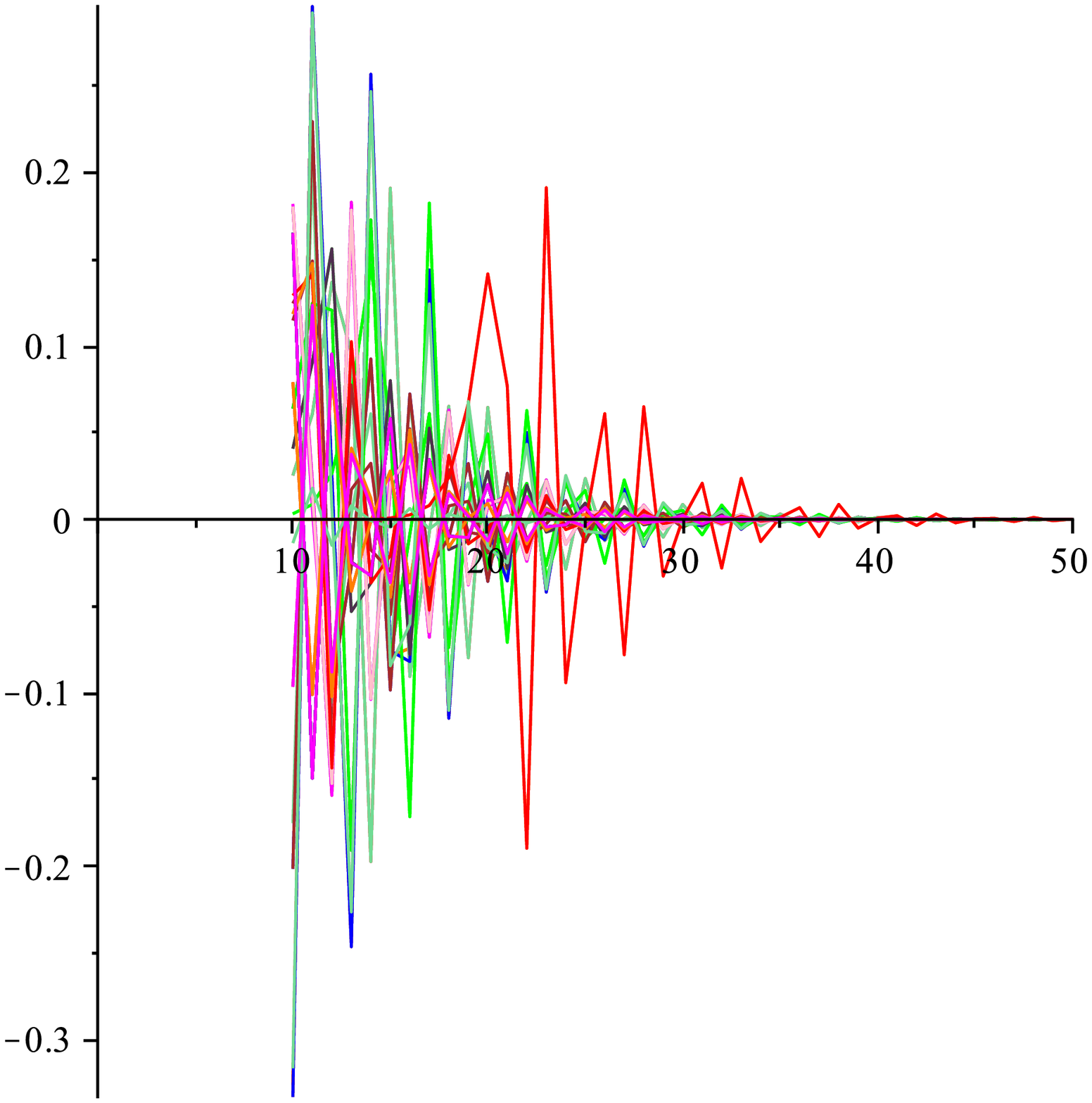}
\caption{The control $u=-\frac{1}{3} \left(f\left(x_{k} \right)-f\left(x_{k-1} \right)\right).$ }
\label{Fig10}
%\end{minipage}
\end{figure}

%Fig. \ref{Fig10}:

%\centerline{ \includegraphics*[width=2.00in, height=2.00in, keepaspectratio=false]{image17.eps}}

\subsection{$T=2$} If $h>3,$ the equilibrium of the logistic equation lost stability and 
2-periodic cycle $(\eta_1,\eta_2)$ emerges, where
$$
\eta_1=\frac{1+h-\sqrt{h^2-2h-3}}{2h},\; \eta_2=\frac{1+h+\sqrt{h^2-2h-3}}{2h}.
$$
The cycle multiplier is  $\mu=-h^2+2h+4.$ For $h\in(3,1+\sqrt6)$ the multiplier decreases from 
$1$ to $-1$. For $h=1+\sqrt6$ the 2-cycle lost stability and 4-cycle emerges. If 
$h\in(1+\sqrt6,4]$ then $\mu\in[-4,-1),$ i.e. $\mu^*=4.$ We want to stabilize 2-cycle for all
 $\mu\in[-4,-1).$ Since $\mu_2(2)=4=\mu^*$ and $\mu_3(2)=9>\mu^*$ then 
 $n^*=3, N^*=(n^*-1)T=4.$ The optimal strength coefficients are
 $\epsilon_1=a_2^0+a_3^0=\dfrac49,\; \epsilon_2=a_3^0=\dfrac19.$ The required stabilizing 
 control is
 $$
 u_k=-\frac{4}{9} \left(f\left(x_{k} \right)-f\left(x_{k-2} \right)\right)
 -\frac{1}{9} \left(f\left(x_{k-2} \right)-f\left(x_{k-4} \right)\right).
 $$
 
%\centerline{ \includegraphics*[width=2.00in, height=2.00in, keepaspectratio=false]{image18.eps}}
%\centerline{ \includegraphics*[width=2.00in, height=2.00in, keepaspectratio=false]{image19.eps}}

\begin{figure}[ht]
%\centering
\begin{minipage}[b]{0.45\linewidth}
\hspace{2cm}\includegraphics[scale=0.15]{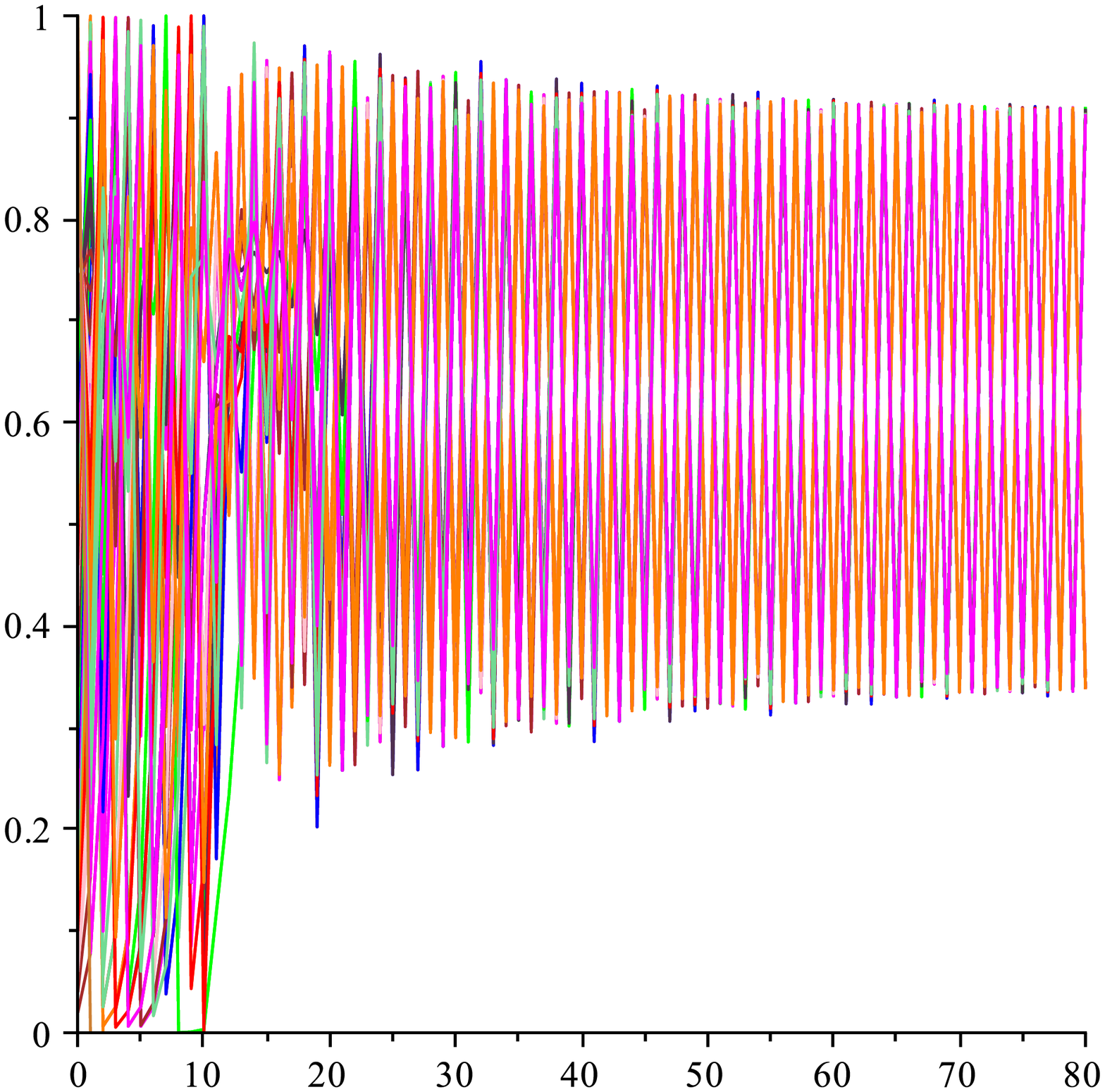}
\caption{ Dynamics of the solutions to the logistic equations for $h=4$  closed by stabilizing control 
$ u_k=-\frac{4}{9} \left(f\left(x_{k} \right)-f\left(x_{k-2} \right)\right)
 -\frac{1}{9} \left(f\left(x_{k-2} \right)-f\left(x_{k-4} \right)\right). $}
\label{Fig11}
\end{minipage}
\;
\begin{minipage}[b]{0.45\linewidth}
\hspace{2cm}\includegraphics[scale=0.15]{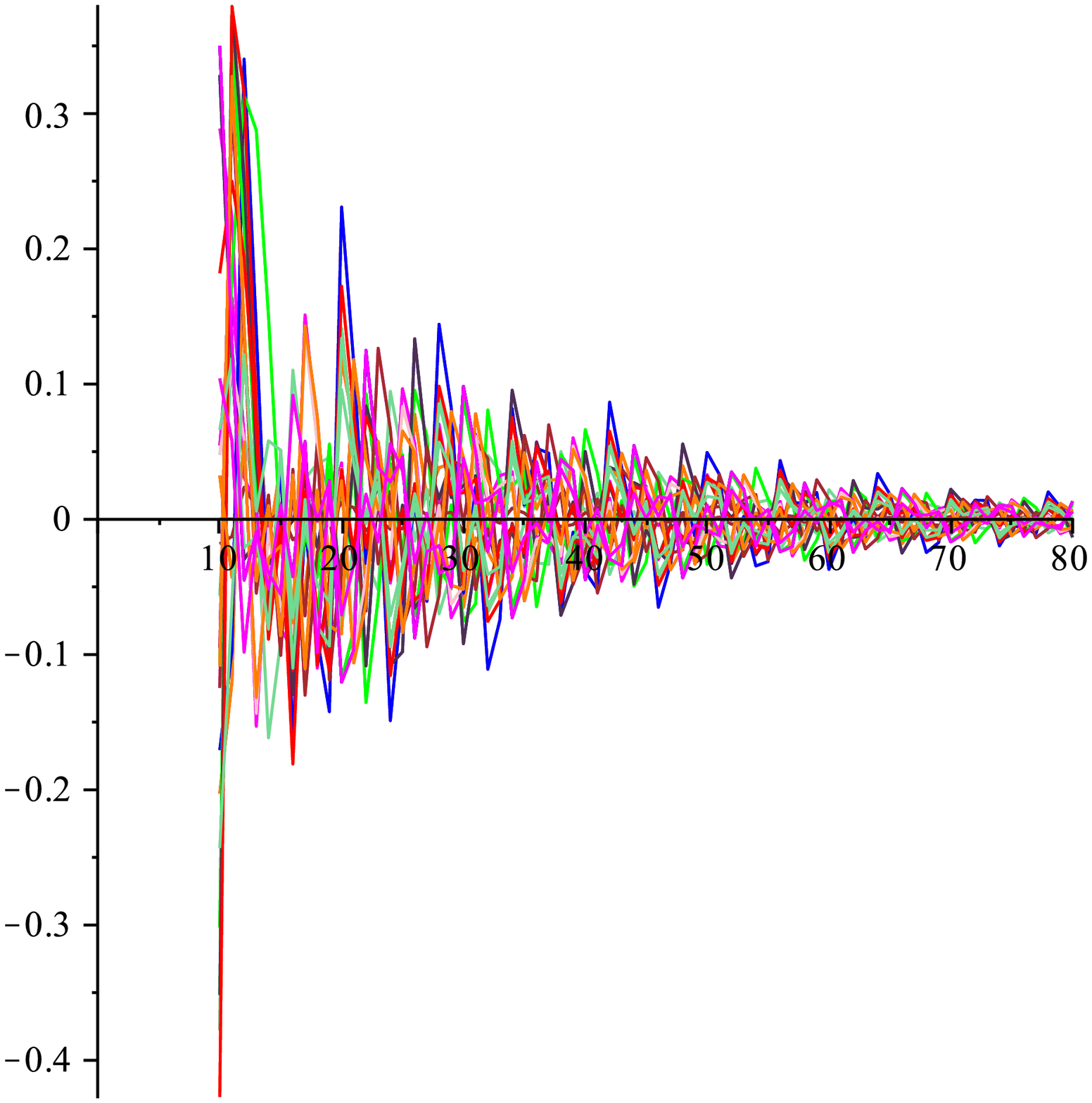}
\caption{ The control $ u_k=-\frac{4}{9} \left(f\left(x_{k} \right)-f\left(x_{k-2} \right)\right)
 -\frac{1}{9} \left(f\left(x_{k-2} \right)-f\left(x_{k-4} \right)\right). $ }
\label{Fig12}
\end{minipage}
\end{figure}

%Fig. \ref{Fig11}: 
%Fig. \ref{Fig12}: 

\section{Searching for cycles of arbitrary length}

\subsection{Algorithm} In the previous section to stabilize cycles of lengths 1 and 2 we used direct DFC 
regulators \eqref{2}. Same regulators can be applied to stabilize cycles of the arbitrary length. Indeed, 
each vale of $T$-cycle is a fixed point for $T$-folded $f$ mapping:
$$
x_{k+1}=f^{(T)}(x_k),\; f^{(T)}(x_k)=f(f^{(T-1)}(x_k)),\; f^{(0)}(x_k)=f(x_k).
$$ 
The question of cycle stability is reducing to the question of stability for the fixed points of the mapping
$f^{(T)},$ that containes cycles of the length $T.$ The value of the cycle multiplier is independent of the 
choice of the fixed point belonging to the considering cycle.

Indeed, 
$$
\mu=\frac d{d\eta}f^{(T)}(\eta_j)=\frac d{d\eta}f(f^{(T-1)}(\eta_j))=
\frac{df}{d\eta}(\eta_{j-1})\frac d{d\eta}f^{(T-1)}(\eta_j)=...
$$
$$
\frac{df}{d\eta}(\eta_{j-1})\cdot \frac{df}{d\eta}(\eta_{j-2})\cdot\ldots\cdot \frac{df}{d\eta}(\eta_{j-T})=
\frac{df}{d\eta}(\eta_{ 1})\cdot \frac{df}{d\eta}(\eta_{2})\cdot\ldots\cdot \frac{df}{d\eta}(\eta_{T})
$$
Note that fixed points of the mapping $f^{(T_1)}$ are also the fixed points of the mapping $f^{(T)},$
if $T=m T_1$ ($m$ - is an integer number).

Thus, the problem of stabilizing $T$-cycles \eqref{1}, \eqref{2} can be reduced to the problem of the stabilizing
the equilibrium for the system
\begin{equation}\label{19}
x_{k+1}=f^{(T)}(x_{k}),\; x_k\in\mathbb R, k=1,2,...
\end{equation}
by the control
\begin{equation}\label{20}
u_k=-\sum_{j=1}^{n-1}\varepsilon_j\left(f^{(T)}(x_{k-2j+2})-f^{(T)}(x_{k-2j}) \right),\; 0<\varepsilon<1,\; j=1,\dots,n-1.
\end{equation}
The solution to the problems \eqref{19}, \eqref{20} is given by Theorem \ref{trm2}:  
$\varepsilon_j=1-\sum_{i=1}^j a_i^0,$ where $a_i^0$ $(i=1,\dots,n-1)$ are defined by \eqref{13}.\\

If $T$ is even, then the problem of stabilization of $T$-cycle \eqref{1}, \eqref{2} can be reduced to to the problem
of stabilizing for the 2-cycle of the system
\begin{equation}\label{21}
x_{k+1}=f^{(T/2)}(x_{k}),\; x_k\in\mathbb R, k=1,2,...
\end{equation}
by the control
\begin{equation}\label{22}
u_k=-\sum_{j=1}^{n-1}\varepsilon_j\left(f^{(T/2)}(x_{k-j+1})-f^{(T/2)}(x_{k-j}) \right),\; 0<\varepsilon<1,\; j=1,\dots,n-1.
\end{equation}
The solution to the problems \eqref{21}, \eqref{22} is given by the Theorem \ref{trm3}:  
$\varepsilon_j=1-\sum_{i=1}^j a_i^0,$ where $a_i^0$ are defined by \eqref{17}.\\

In general case, when $T=m T_1$ ($m$ - is an integer number) the problem of the $T$-cycle stabilization is reducing to the problem of the stabilizing of the cycle of the length $T_1$ of the auxiliary system
\begin{equation}\label{23}
x_{k+1}=f^{(m)}(x_{k}),\; x_k\in\mathbb R, k=1,2,...
\end{equation}
by the control
\begin{equation}\label{24}
u_k=-\sum_{j=1}^{n-1}\varepsilon_j\left(f^{(m)}(x_{k-T_1j+1})-f^{(m)}(x_{k-T_1j}) \right),\; 0<\varepsilon_j<1,\; j=1,\dots,n-1.
\end{equation}

Let us note the important difference between the suggested  method of stabilization in this article and majority of currently known ones (in particular, from
famous OGY method \cite{OGY}). Namely, the control is applying at any instance of time, not necessary in a neighborhood of the desired cycle. It is 
{\it not necessary} to know the cycle a priori. Moreover, one control allows to stabilize at once {all} cycles of the given length with negative multipliers (with a proper choice of the number $n$ of strength coefficients). To stabilize the specific cycle it is sufficient for the sequence of the initial points to fall
into basin of attraction that corresponds to the fixed point of the mapping $f^{(T)}$.    

The problems \eqref{1}, \eqref{2} and \eqref{19}, \eqref{20} seem  to be equivalent
at the first glance. But this is not the case. The basins of attractions
of the investigated  equilibriums or cycles in this problems, generally speaking, are different. This is why we need to have a collection of controls that stabilize cycles of the lengths $1,2,...,T$ for all $n.$ With the help of this controls it can be possible to stabilize the cycles of the arbitrary length. By varying parameters $m$ and
$T_1\in\{1,2,...,\hat T\}$ it might be possible to increase the basin of attraction for the investigating cycles.

One of the possible applications of the suggested methods could be verification of the existence of the periodic orbits of a given nonlinear mapping. If such orbits are non-stable and their multipliers are negative, then such orbits can be detected by stabilization. It is possible to find all unstable  periodic orbits with negative multipliers with sufficiently large basins of attraction. 

Let the system contains $s$ cycles with unknown multipliers $\{\mu_1,\dots,\mu_s\}\subset [-\mu,-1),$ and we don't know the value of $s.$ However, it is not difficult to get an estimate for the value $\mu^*$ from the properties of the derivative of the initial mapping or from the $T$ iteration of this mapping. Then the number $n$ in the stabilizing control, which depends on $\mu^*$ can be computed from the condition $m_n(T)>\mu^*,\; m_{n-1}(T)\le \mu^*.$ (The functions $\mu_n(1)$ and $\mu_n(2)$ are obtained in the theorems 2 and 3. ) On practice, there is no need to find the value $n$ - it might be consecutively increasing until new stable cycles stop to appear.

\subsection{Examples} It is known that the logistic mapping \eqref{18} for $h=4$ has two 3-cycles: $\{\eta_1\approx 0.188, \eta_2\approx 0.611, \eta_3\approx 0.95\}$ and $\{\eta_1\approx 0.117, \eta_2\approx 0.413, \eta_3\approx 0.97\}.$ The first cycle has a positive multiplier and the second one has a negative one. Let us use the suggested methodology to find the second 3-cycle. To do that, consider the system
\begin{equation}\label{25}
x_{k+1}=f^{(3)}(x_k),\; x_k\in\mathbb R,\; k=1,2,...
\end{equation}
closed by the control
\begin{equation}\label{26}
u_k=-\sum_{j=1}^{n-1} \varepsilon_j\left(f^{(3)}(x_{k-j+1}) - f^{(3)}(x_{k-j})\right),\; 0<\varepsilon_j<1,\; j=1,...,n-1.
\end{equation}
For $n=4$ the system \eqref{25}, \eqref{26} demonstrates three stable equilibriums $x_1^*\approx 0.117, x_2^*\approx 0.413, x_3^*\approx 0.75$ (see Fig. \ref{Fig13})
The first and the second one correspond  to the non-stable 3-cycle of the logistic mapping with negative multiplier, and the third one corresponds to the unstable equilibrium. The stabilizing control vanish with increasing of the number of iterations $k$ (see Fig. \ref{Fig14}).

%\centerline{ \includegraphics*[width=2.00in, height=2.00in, keepaspectratio=false]{image25.eps}}

%Fig.13 Stabilization of the 3-cycle of the logistic equations by stabilizing  the equilibrium \eqref{25} by the control \eqref{26}.

%\centerline{ \includegraphics*[width=2.00in, height=2.00in, keepaspectratio=false]{image26.eps}}

%Fig.14 The control \eqref{26}.\\

\begin{figure}[ht]
%\centering
\begin{minipage}[b]{0.45\linewidth}
\hspace{2cm}\includegraphics[scale=0.15]{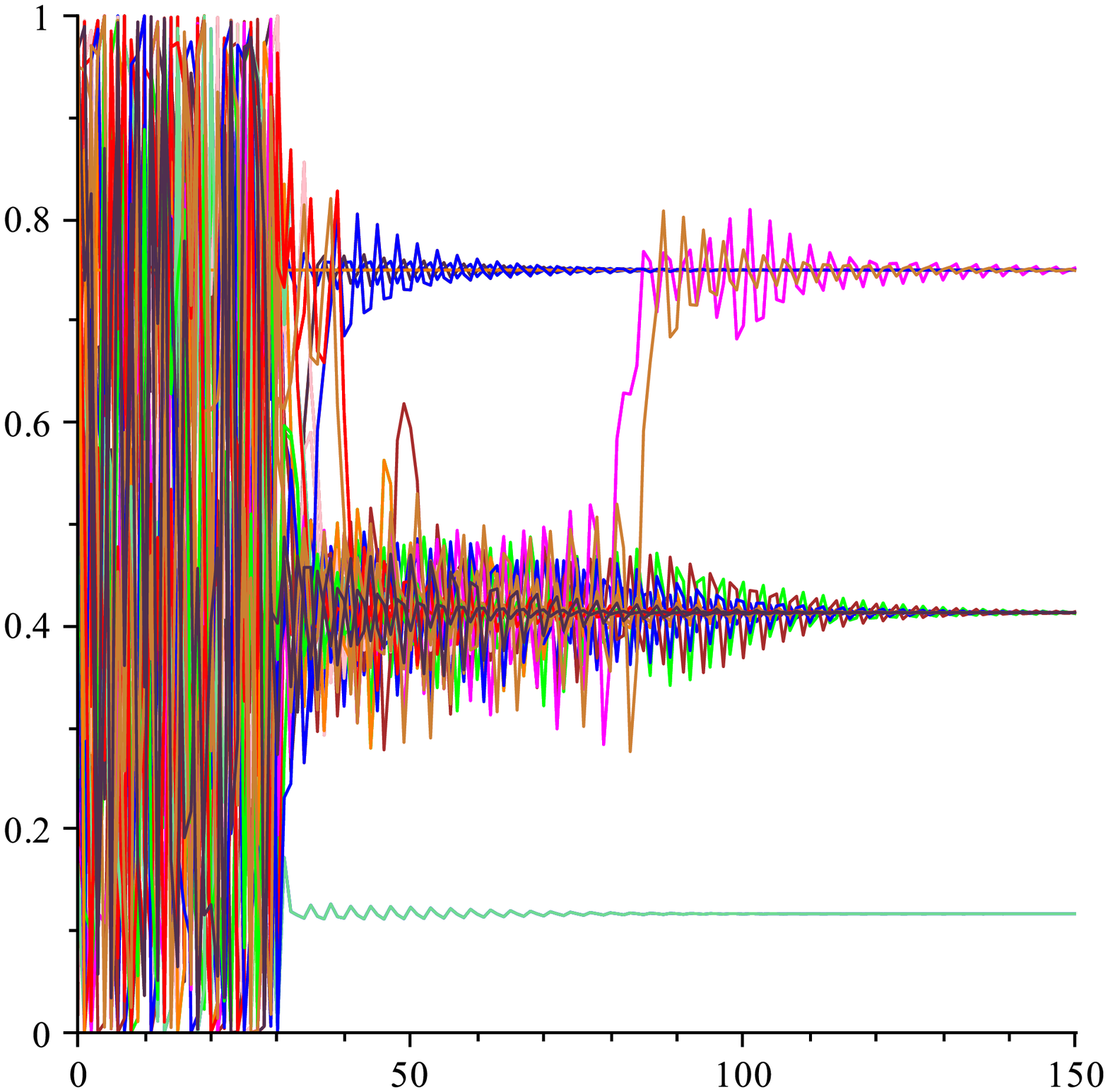}
\caption{Stabilization of the 3-cycle of the logistic equations by stabilizing  the equilibrium \eqref{25} by the control \eqref{26}.  }
\label{Fig13}
\end{minipage}
\;
\begin{minipage}[b]{0.45\linewidth}
\hspace{2cm}\includegraphics[scale=0.15]{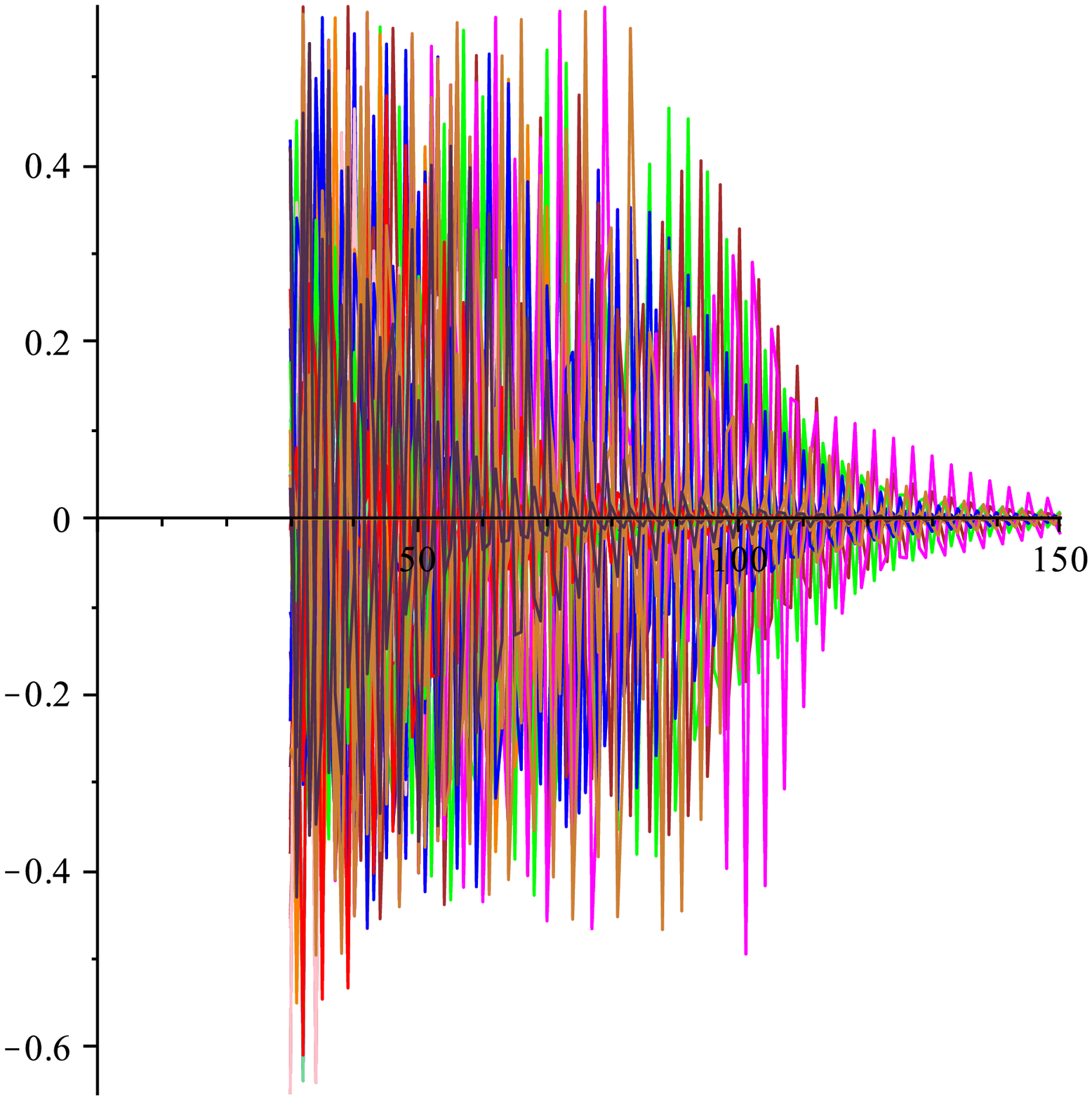}
\caption{ The control \eqref{26}. }
\label{Fig14}
\end{minipage}
\end{figure}

%Fig.  \ref{Fig14}:
%Fig. \ref{Fig13}:  \\

Same way was conducted to search for the cycles of the length 7. It should be mentioned that it is not always true that the trajectory quickly arrived at the basin of attraction of the fixed point of the mapping (see Fig 15) and not for every initial point the stabilizing control start to decrease rapidly (see Fig. 16).

%\centerline{ \includegraphics*[width=2.00in, height=2.00in, keepaspectratio=false]{image27.eps}}

%Fig.15 Stabilization of the 7-cycle of the logistic equations by stabilizing  the equilibrium of the auxiliary equation.
 
%\centerline{ \includegraphics*[width=2.00in, height=2.00in, keepaspectratio=false]{image28.eps}}

%Fig.16 The control stabilizing the cycles of the length 7 of the logistic equation $(n=22).$\\

\begin{figure}[ht]
\centering
\begin{minipage}[b]{0.45\linewidth}
\hspace{2cm}\includegraphics[scale=0.15]{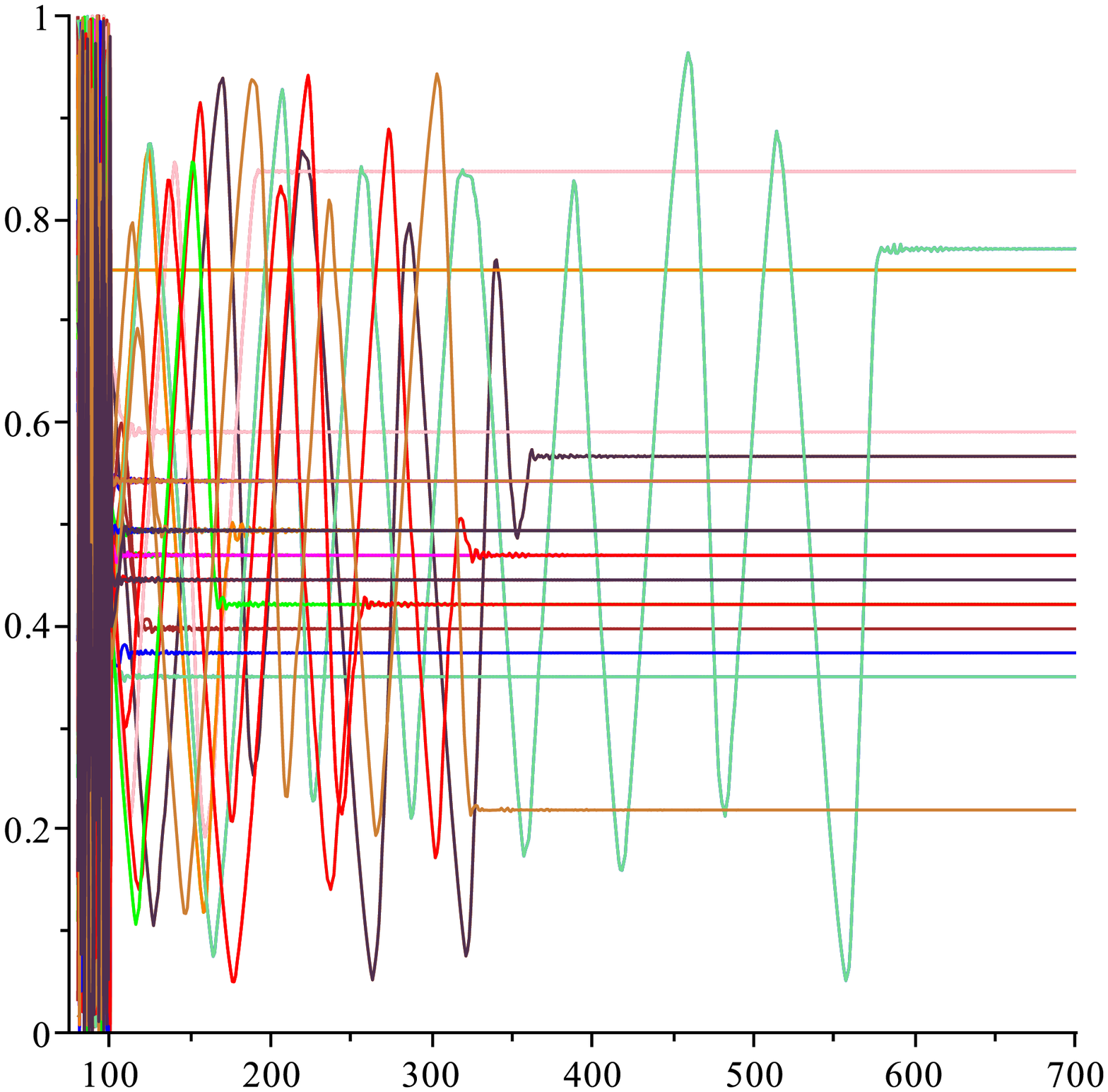}
\caption{Stabilization of the 7-cycle of the logistic equations by stabilizing  the equilibrium of the auxiliary equation.  }
\label{Fig15}
\end{minipage}
\;
\begin{minipage}[b]{0.45\linewidth}
\hspace{2cm}\includegraphics[scale=0.15]{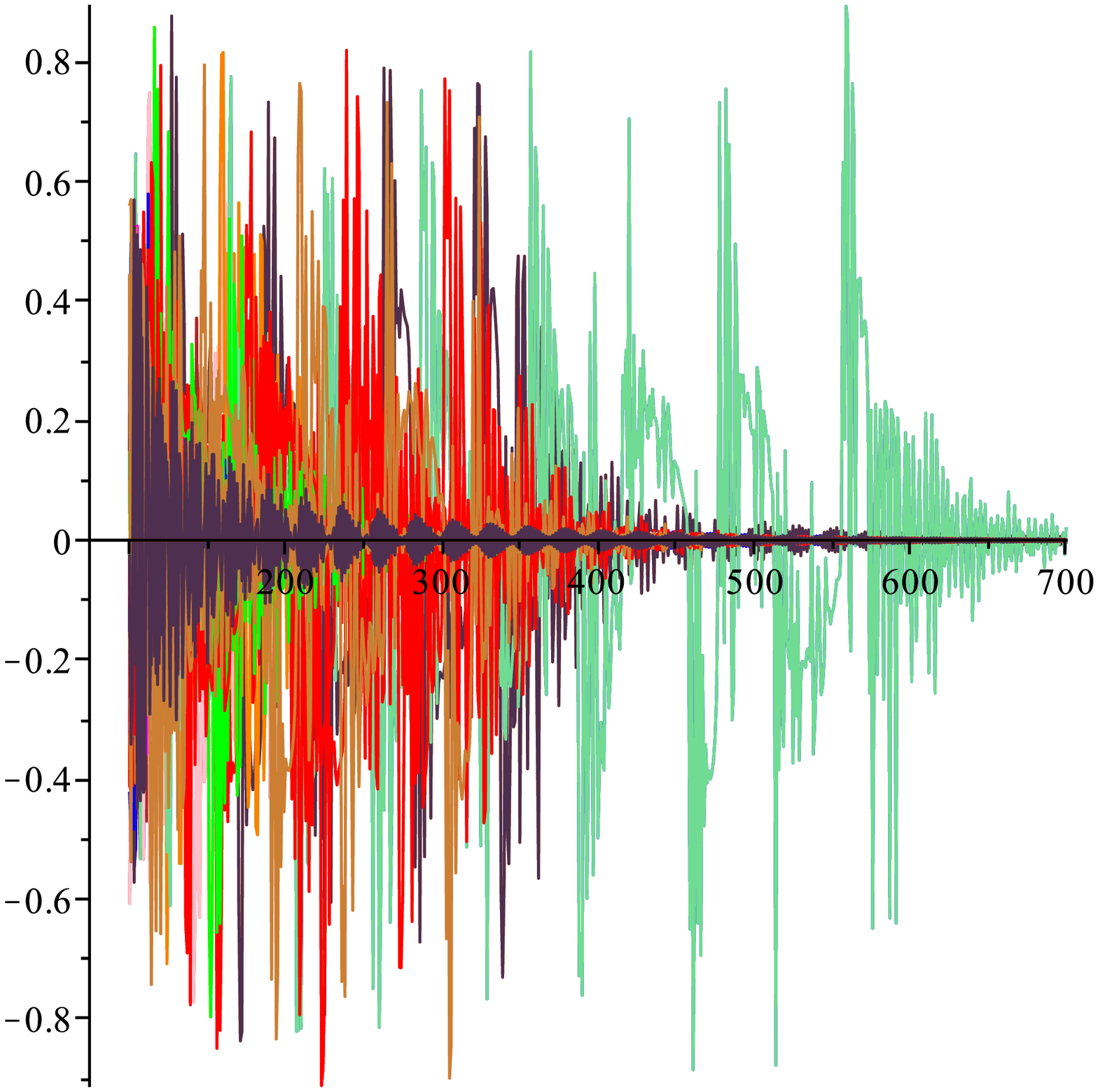}
\caption{The control stabilizing the cycles of the length 7 of the logistic equation $(n=22).$ }
\label{Fig16}
\end{minipage}
\end{figure}

%Fig. \ref{Fig15}:   
%Fig. \ref{Fig16}:   \\

A little bit different situation can be observed in the stabilizing of the cycles of the length 8. Besides the equilibriums of the auxiliary  systems that corresponds to the actual cycles of the lengths 8, one can see a stabilization of the equilibria corresponding to the cycles of the lengths 1,2,4 (see Fig. 17). It can be 
concluded from the  analysis of the  behavior of the graph for the corresponding control  (se Fig. 18). Namely, not for every initial value the control is vanishing. For some initial values the control approaching the 2-cycles and 4-cycles.
%\centerline{ \includegraphics*[width=2.00in, height=2.00in, keepaspectratio=false]{image29.eps}}
%\centerline{ \includegraphics*[width=2.00in, height=2.00in, keepaspectratio=false]{image30.eps}}

\begin{figure}[ht]
\centering
\begin{minipage}[b]{0.45\linewidth}
\hspace{2cm}\includegraphics[scale=0.15]{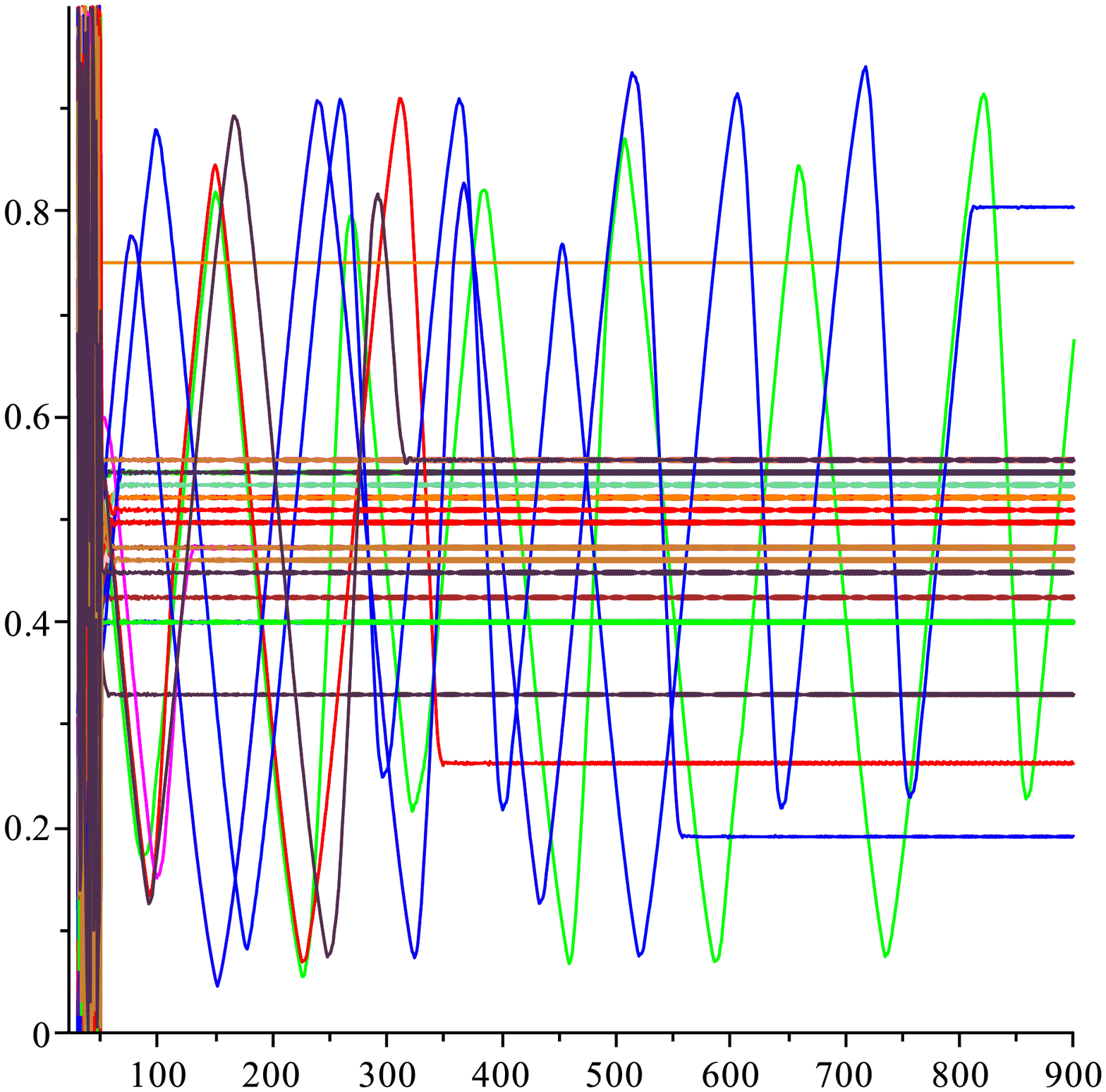}
\caption{  Stabilization of the 8-cycle of the logistic equations by stabilizing  the equilibrium of the auxiliary equation. }
\label{Fig17}
\end{minipage}
\;
\begin{minipage}[b]{0.45\linewidth}
\hspace{2cm}\includegraphics[scale=0.15]{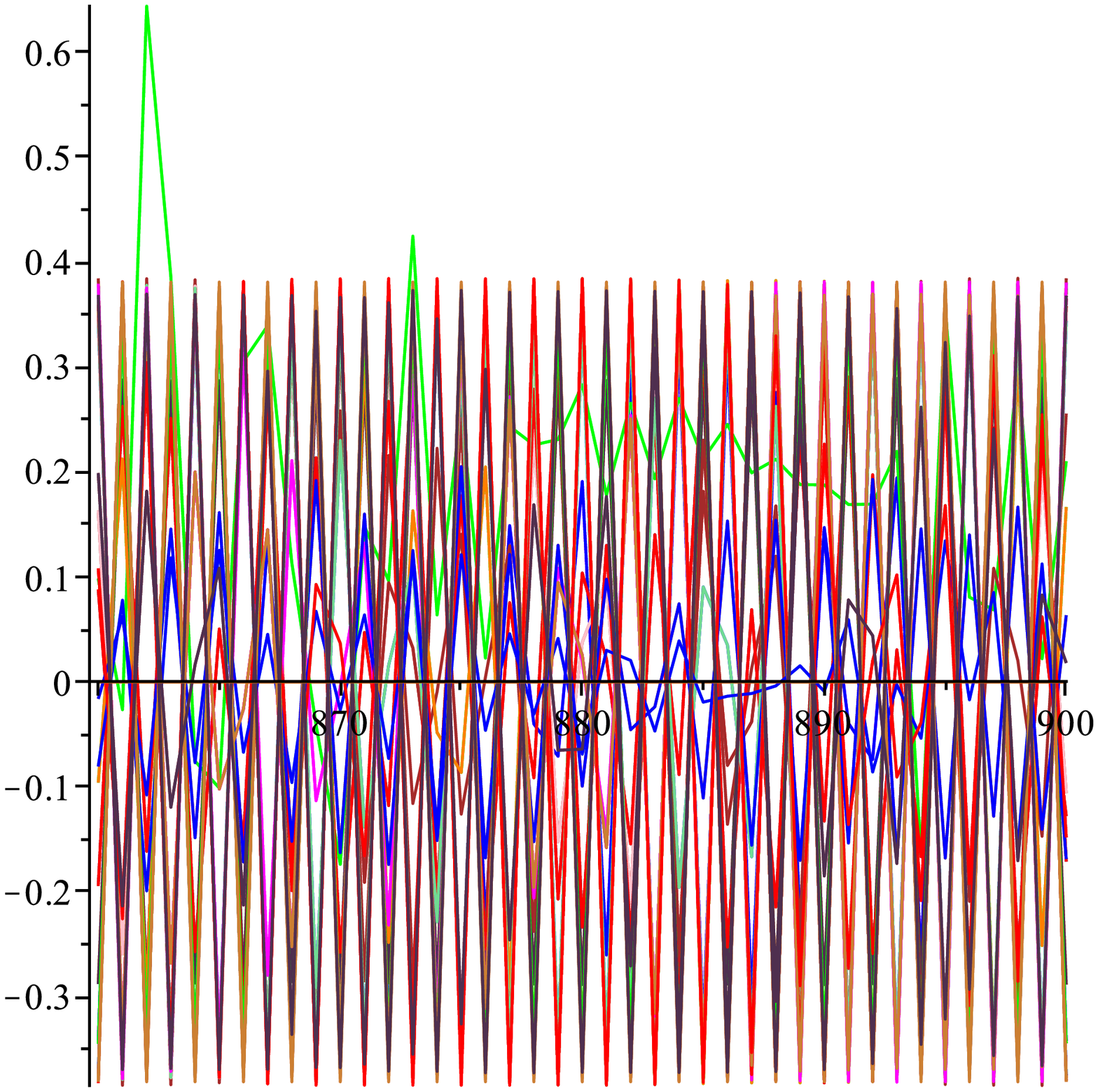}
\caption{ The control stabilizing the cycles of the length 8 of the logistic equation $(n=40).$ }
\label{Fig18}
\end{minipage}
\end{figure}

%Fig. \ref{Fig17}:  
%Fig. \ref{Fig18}: 

For the stabilization of 8-cycles  the following auxiliary systems turns out to be more effective. 
\begin{equation}\label{27}
x_{k+1}=f^{(4)}(x_k),\; x_k\in\mathbb R,\; k=1,2,...
\end{equation}
closed by the control
\begin{equation}\label{28}
u_k=-\sum_{j=1}^{n-1} \varepsilon_j\left(f^{(4)}(x_{k-2j+2}) - f^{(4)}(x_{k-2j})\right),\; 0<\varepsilon_j<1,\; j=1,...,n-1.
\end{equation}

To stabilize the 2-cycles for the system \eqref{27}, \eqref{28} it is enough to take $n=19$ (see Fig. 19). In this case the stabilizing control does vanish
(see Fig. 20).

%\centerline{ \includegraphics*[width=2.00in, height=2.00in, keepaspectratio=false]{image31.eps}}
%\centerline{ \includegraphics*[width=2.00in, height=2.00in, keepaspectratio=false]{image32.eps}}
\begin{figure}[ht]
\centering
%\begin{minipage}[b]{0.45\linewidth}
\includegraphics[scale=0.15]{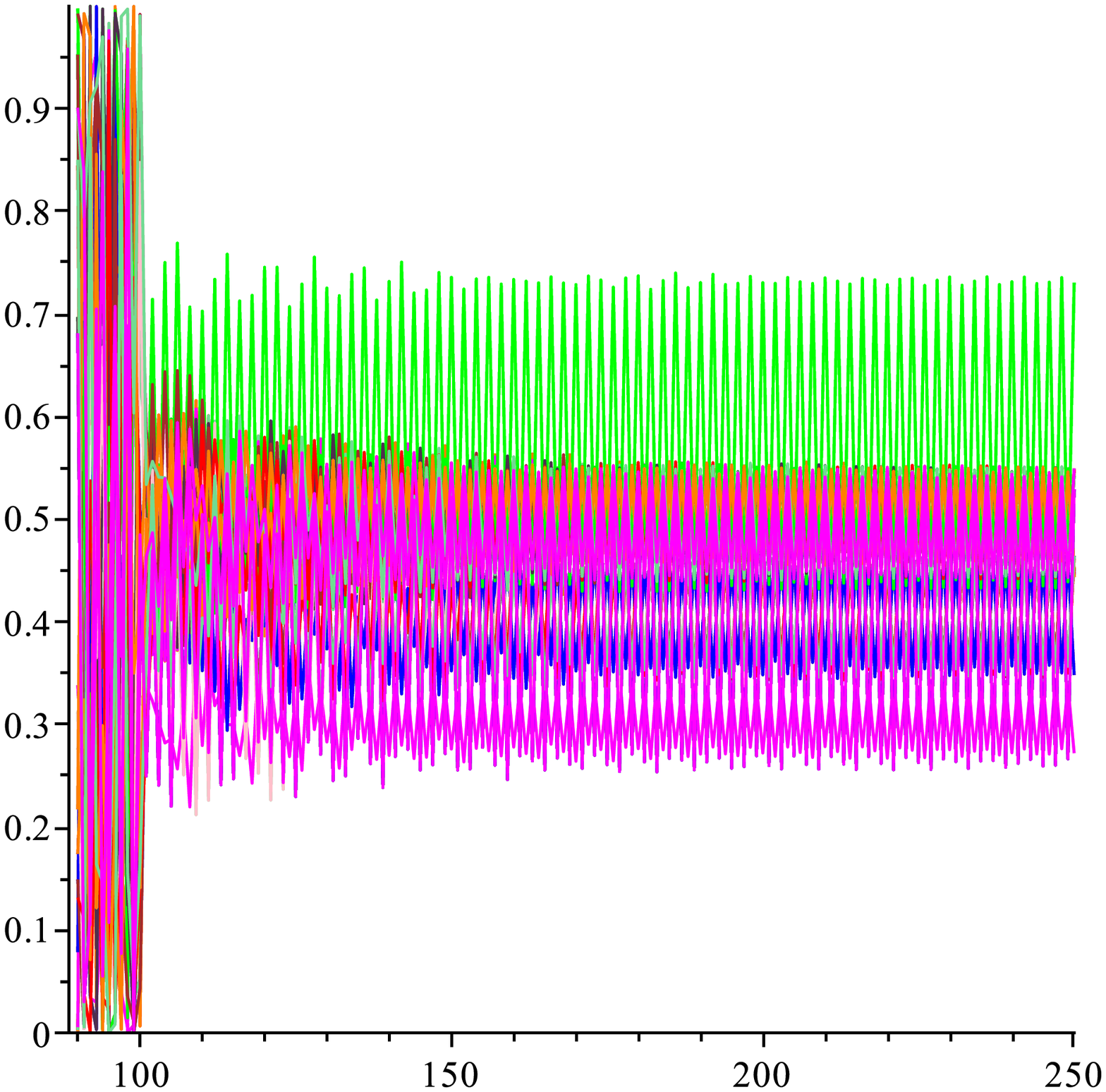}
\hspace{.1cm}
\includegraphics[scale=0.15]{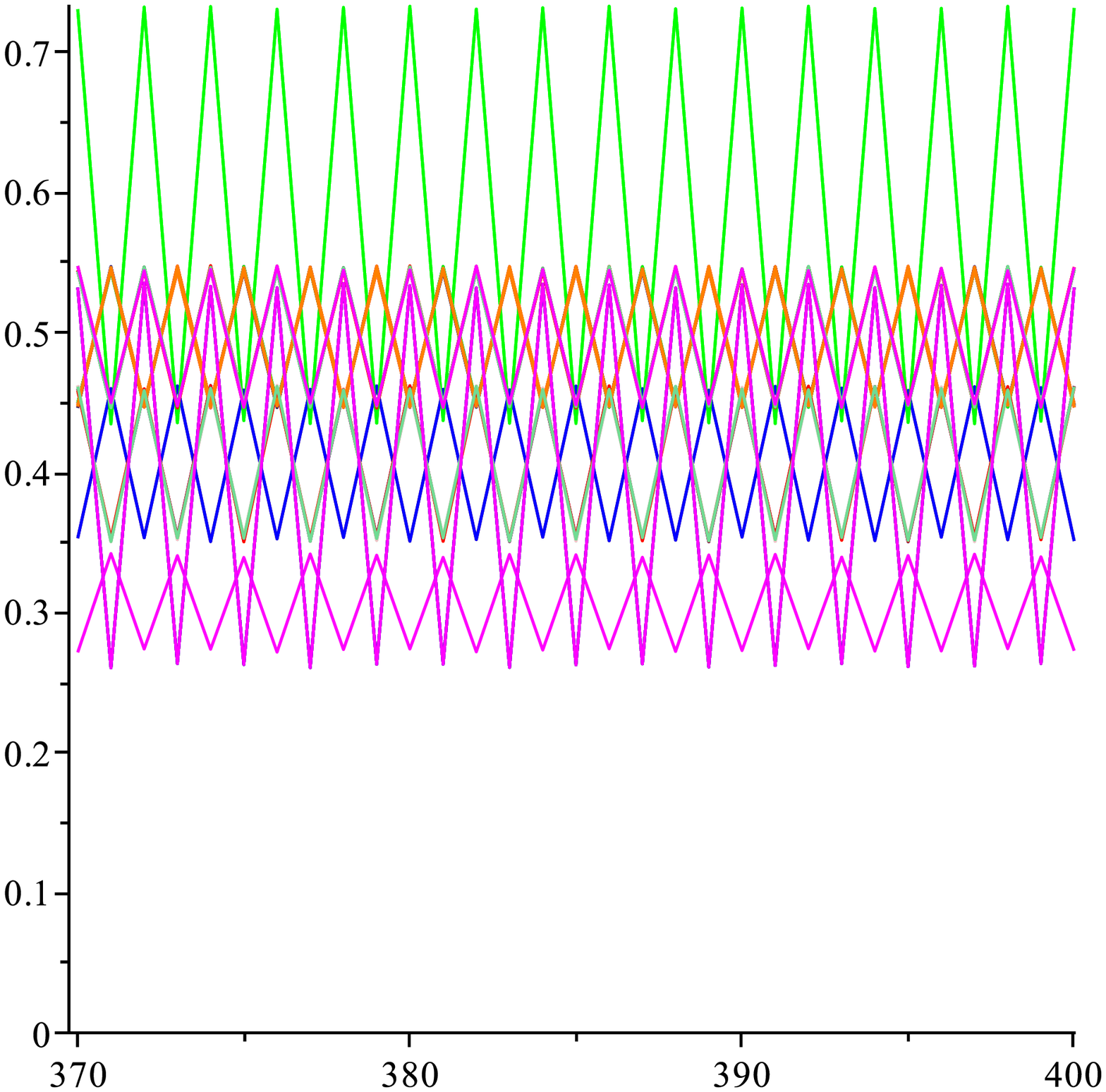}
\caption{ Stabilization of the 8-cycle of the logistic equations by stabilizing  the 2-cycle of the auxiliary equations \eqref{27} by the control \eqref{28}.  }
\label{Fig19}
%\end{minipage}
%\;
%\begin{minipage}[b]{0.45\linewidth}
%\includegraphics[scale=0.15]{image32.eps}
%\caption{ }
%\label{Fig20}
%\end{minipage}
\end{figure}

%Fig. \ref{Fig19}:
 
 \begin{figure}[ht]
\centering
%\begin{minipage}[b]{0.45\linewidth}
\hspace{2cm}\includegraphics[scale=0.15]{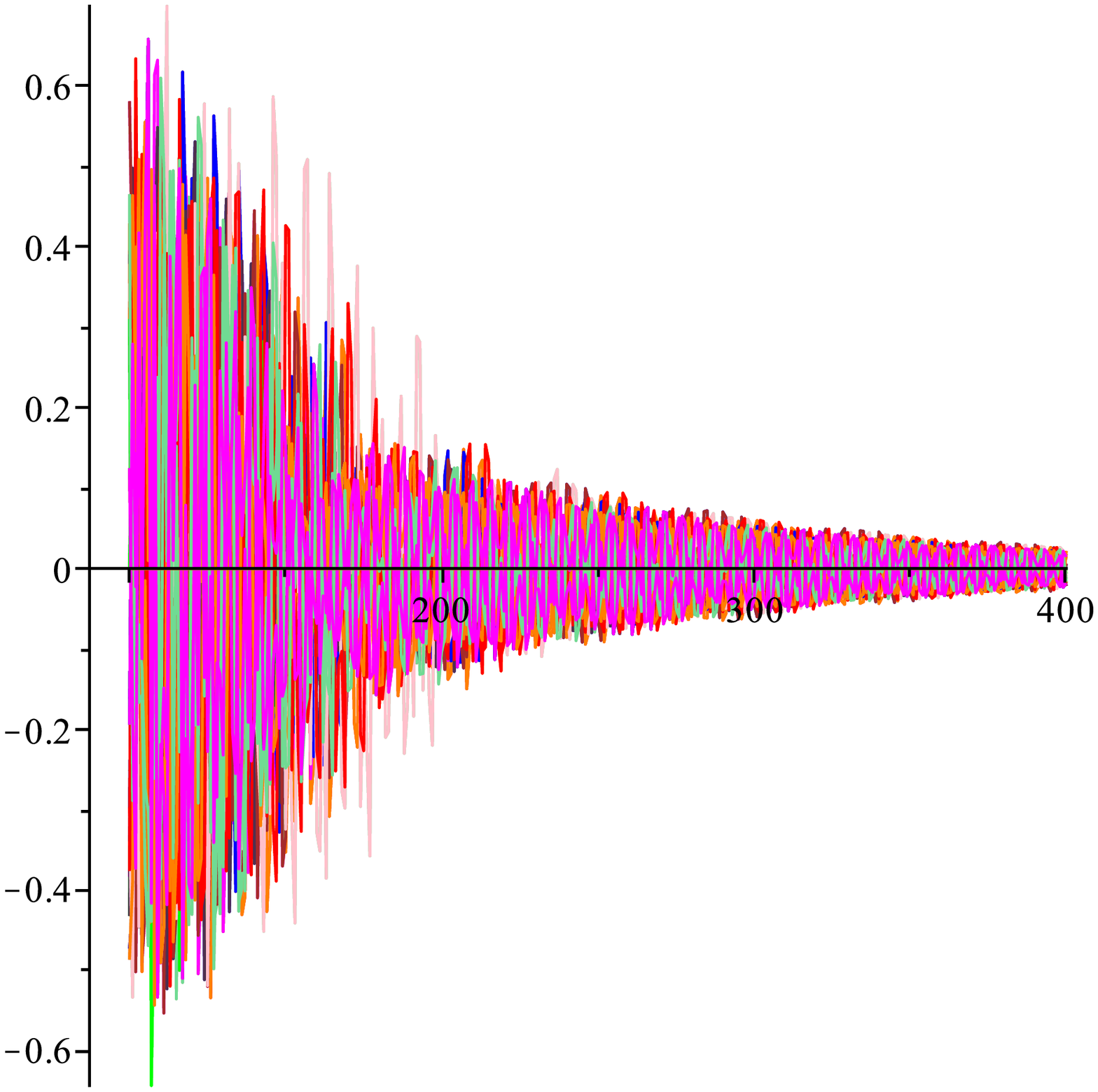}
\caption{The control \eqref{28}.  }
\label{Fig20}
%\end{minipage}
\end{figure}

% Fig. \ref{Fig20}: \\
%\centerline{ \includegraphics*[width=2.00in, height=2.00in, keepaspectratio=false]{image33.eps}}

One of the traditional methods of cycle stabilization in non-linear discrete systems is predictive control schedule \cite{Po}. In \cite{TD} the predictive control
was used for stabilization of the cycles of the length from 1 to 6 in the system with sudden occurrence of chaos, abbreviated SOC:
$$
f(x)=\frac{ha}2-ha\left|x-\frac12\right|+x,\quad -1\le ha\le1+\sqrt 2.
$$
Below we will illustrate the application of the DFG method for the stabilization of 3-cycle (Fig 21,22) and 7-cycles (Fig 23, 24) in a system with SOC.

%\centerline{ \includegraphics*[width=2.00in, height=2.00in, keepaspectratio=false]{image34.eps}}
%\centerline{ \includegraphics*[width=2.00in, height=2.00in, keepaspectratio=false]{image35.eps}}

\begin{figure}[ht]
\centering
\begin{minipage}[b]{0.45\linewidth}
\hspace{2cm}\includegraphics[scale=0.15]{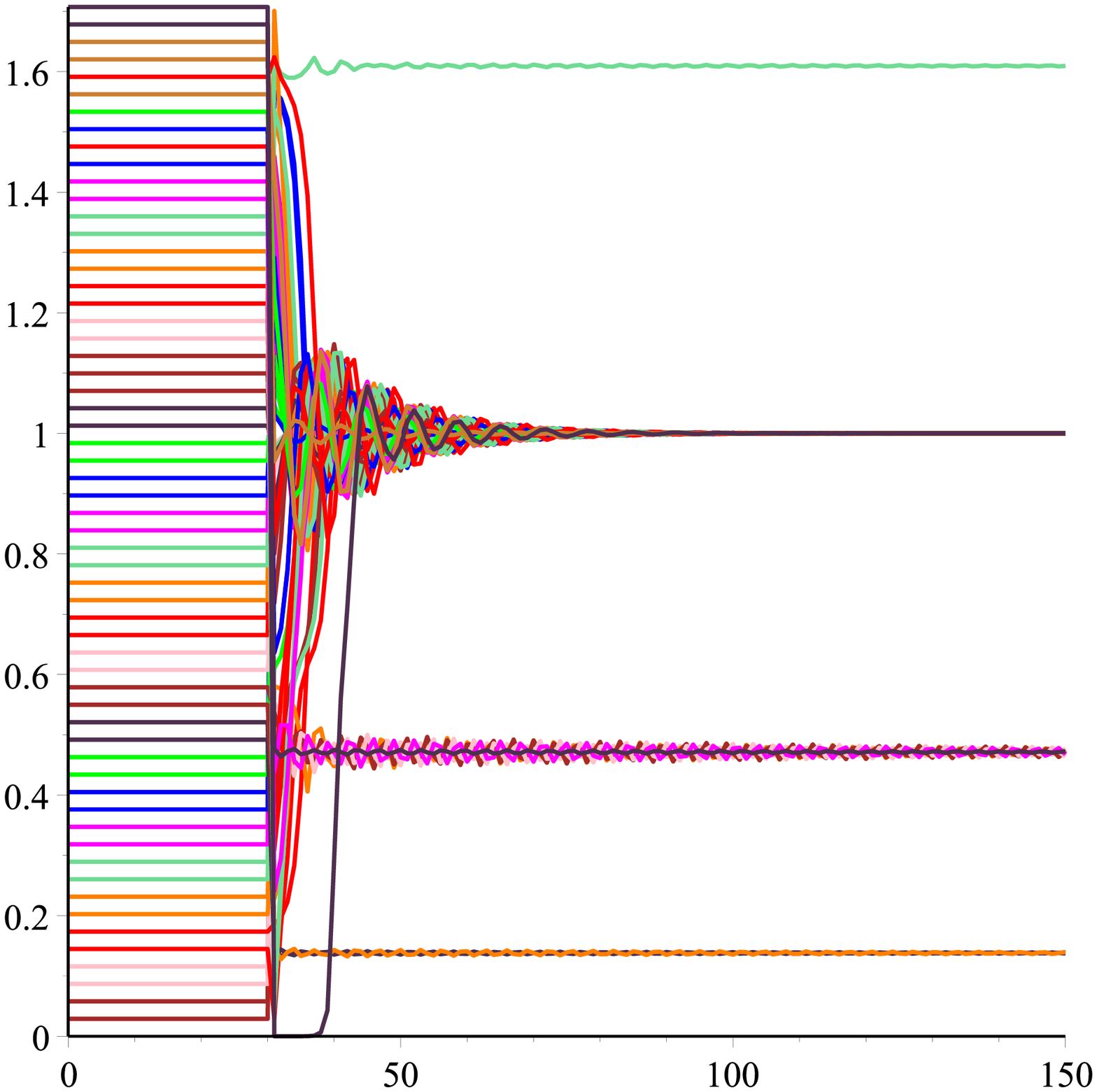}
\caption{ Stabilization of 3-cycle for SOC system by stabilizing  the equilibrium of the auxiliary equation by the DFC (n=7).  }
\label{Fig21}
\end{minipage}
\;
\begin{minipage}[b]{0.45\linewidth}
\hspace{2cm}\includegraphics[scale=0.15]{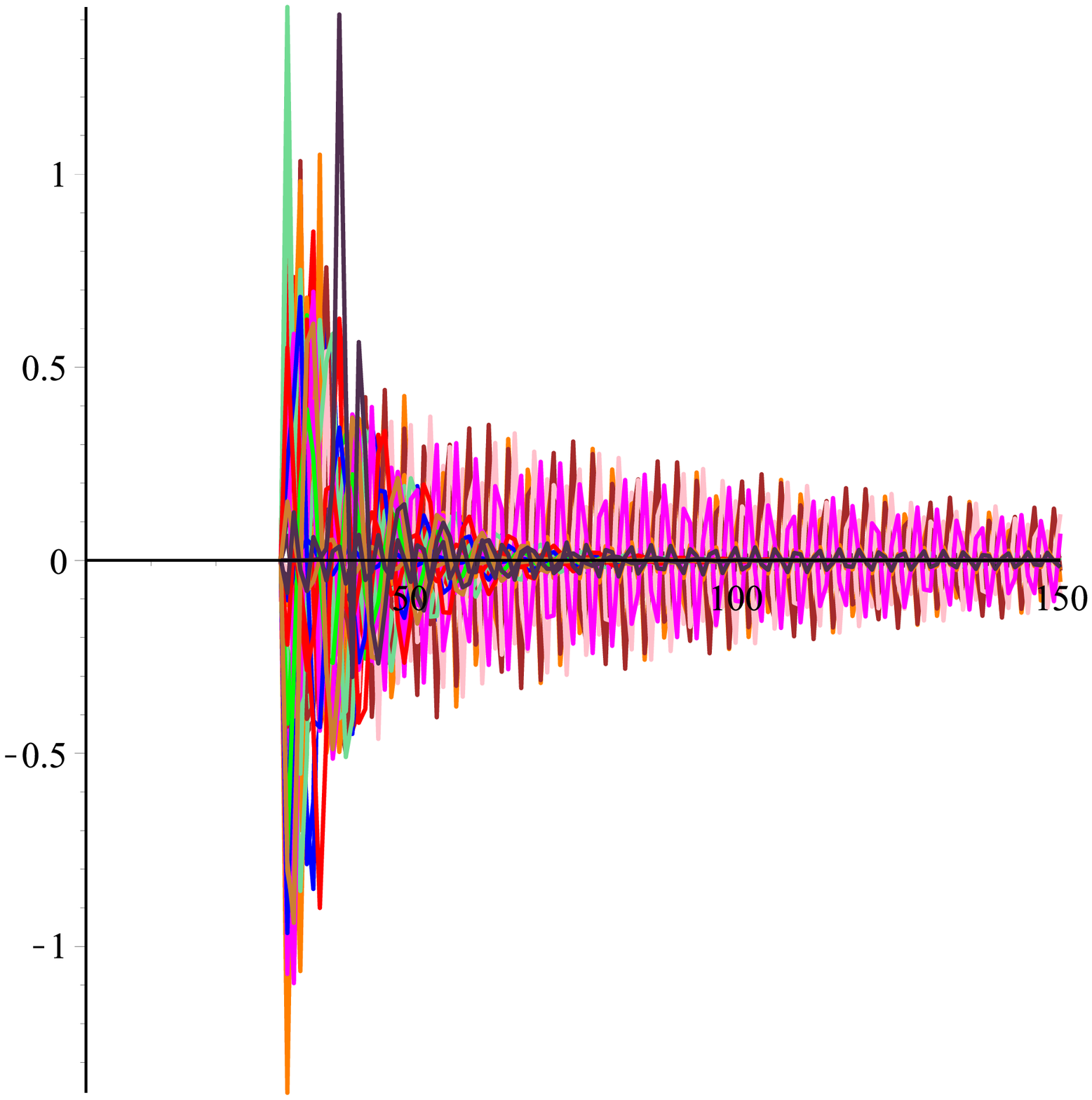}
\caption{ The control stabilizing 3-cycle for SOC system. }
\label{Fig22}
\end{minipage}
\end{figure}

 %Fig. \ref{Fig21}: 
 %Fig. \ref{Fig22}: 
%\centerline{ \includegraphics*[width=2.00in, height=2.00in, keepaspectratio=false]{image36.eps}}
%\centerline{ \includegraphics*[width=2.00in, height=2.00in, keepaspectratio=false]{image37.eps}}

\begin{figure}[ht]
\centering
\begin{minipage}[b]{0.45\linewidth}
\hspace{2cm}\includegraphics[scale=0.15]{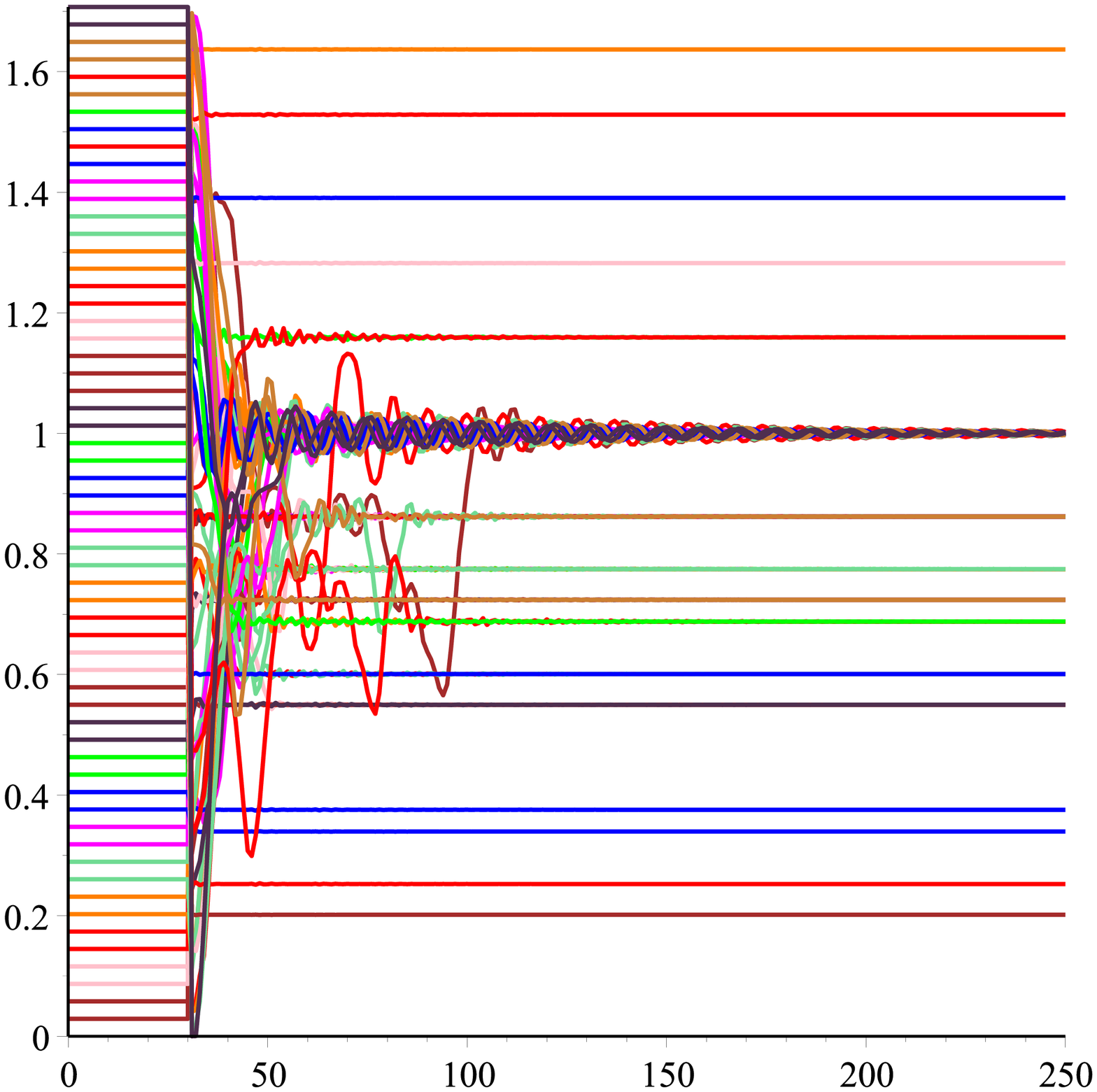}
\caption{Stabilization of 7-cycle for SOC system by stabilizing  the equilibrium of the auxiliary equation by the DFC (n=14).  }
\label{Fig23}
\end{minipage}
\;
\begin{minipage}[b]{0.45\linewidth}
\hspace{2cm}\includegraphics[scale=0.15]{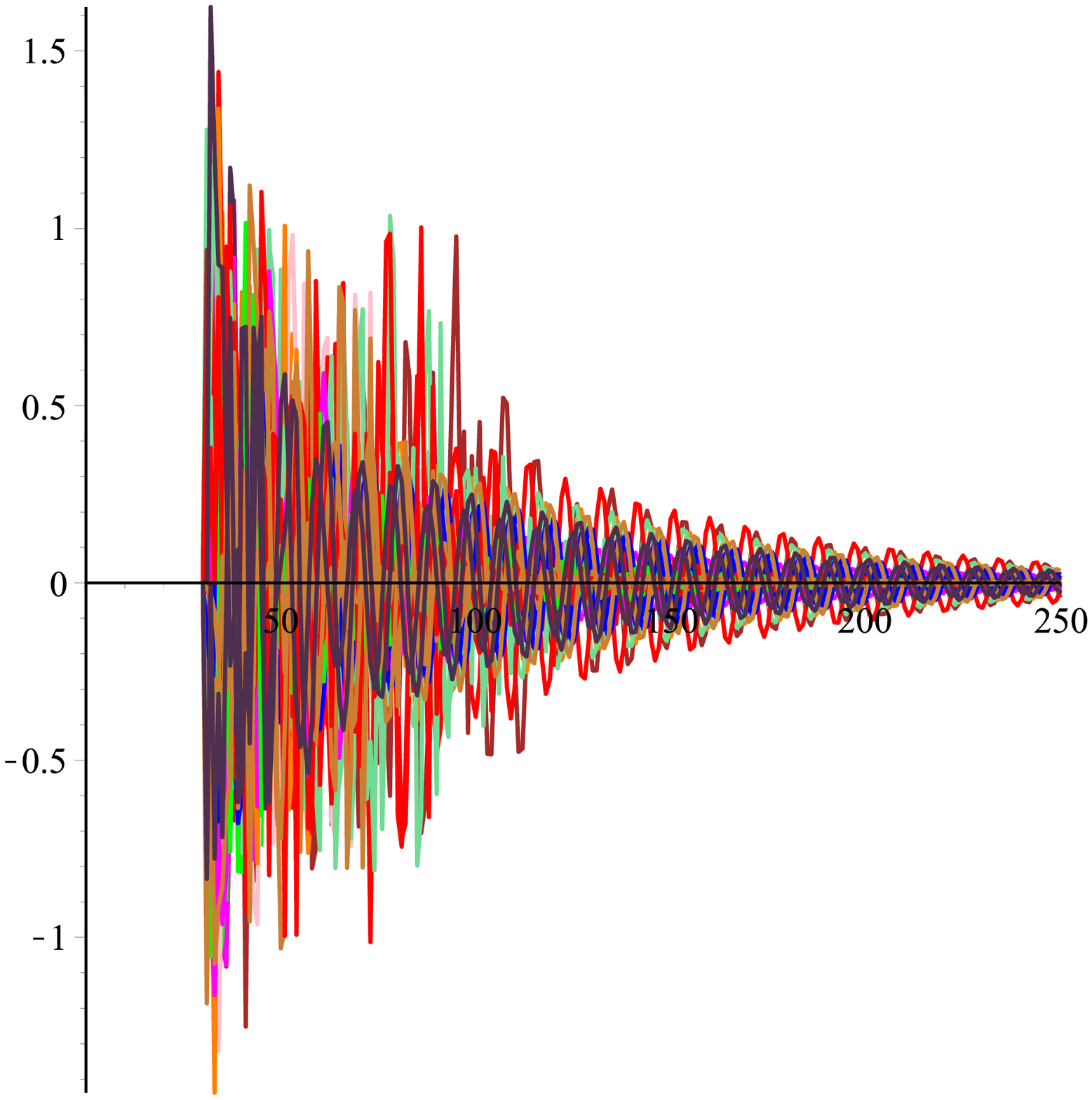}
\caption{The control stabilizing 7-cycle for SOC system. }
\label{Fig24}
\end{minipage}
\end{figure}

% Fig. \ref{Fig23}:  
 %Fig. \ref{Fig24}:  

\subsection{Practical aspects} Let us mention several advantages of the proposed modified DFC method in the article  compare to
the predictive control. The system with predictive control in fact is a system with advancing argument. In such systems it happen  very often that
cause-effect relations are broken, that can lead to various non-controlling effect.  A system with modified  DFC is a system with
delay, the control uses not predicted values rather real. Moreover, instead of the values of the  previous instances of time one needs to know just 
the dispersion of the values of the function in a prior cycle. As a result, the rate of the convergence is increasing and the total number of computations 
is decreasing. For instance, to stabilize the 7-cycle of the logistic equation the predictive control requires 10000 iterations while MDFC just 700. 
Predictive control has limitations on the length of stabilizing cycles, which is caused by the accumulation of the rounding errors of the computational procedures.  MDFC is robust to the parameters of the systems and to the cycle multipliers. Moreover, the rate of robustness can be improved.

Let us explain this moment more precisely. Consider the control \eqref{2} where the strength coefficients are defined by $\varepsilon_j=1-\sum_{i=1}^j a_i^0,$ and $a_i^0, i=1,...,n-1$ are computed by \eqref{13}. This control will stabilize all equilibriums of the system \eqref{1} with multipliers from the set which is inverse with respect to the unit circle to the image of the exterior of the unit disc $\mathbb D$ under the polynomial mapping
\begin{equation}\label{29}
F(z)=\sum_{j=1}^n a_j^0 z^j,
\end{equation}
i.e.
$$
\left\{ 
\mu: \frac1\mu \in F(\bar{\mathbb C}\backslash\mathbb D)
\right\}.
$$
Here $\bar{\mathbb C}$ is the extended complex plane.
The image of the exterior of the unit disc and the inverse sets are displayed on the Fig. \ref{Fig25}.
\begin{figure}[ht]
\centering
%\begin{minipage}[b]{0.45\linewidth}
\hspace{1cm}\includegraphics[scale=0.25]{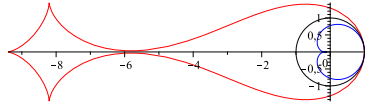}
\caption{The image and  the exterior of the unit disc 
by the polynomial mapping \eqref{29}  and the inversion with respect to the unit circle  (with $n=5$) }
\label{Fig25}
%\end{minipage}
\end{figure}

% Fig. \ref{Fig25}: The image and  the exterior of the unit disc 
%by the polynomial mapping \eqref{29}  and the inversion with respect to the unit circle  (with $n=4$).\\
%\centerline{ \includegraphics*[width=2.00in, height=2.00in, keepaspectratio=false]{image39.eps}}

The inverse image in the neighborhood of some critical points could be tangent to the negative real axis, which
might negatively affect  the robust properties of the control if the values of multipliers are close to these critical 
points. Since the multipliers usually are not known in advance then such situation cannot be excluded.  Therefore,
 it is necessary to separate this set from the real line. It can be done by applying the following procedure which we call  $\varepsilon$-trick. Namely, define new coefficients of the control $\varepsilon_j=1-\sum_{i=1}^j a_i^\varepsilon,$ where
\begin{equation}\label{30}
a_1^\varepsilon=\frac{a_1^0+\varepsilon}{1+\varepsilon},\;a_2^\varepsilon=\frac{a_2^0+\varepsilon}{1+\varepsilon},\; \dots,
a_n^\varepsilon=\frac{a_n^0+\varepsilon}{1+\varepsilon},\quad \varepsilon>0.
\end{equation}
It is clear that the normalizing conditions $\sum_{i=1}^n a_i^\varepsilon =1$ is still valid and the control still stabilize the system. 
In this way the boundary of the region of location of admissible multipliers is shifted off the real axis. On the other hand, the general
linear size of the region decreases in this case. The  Figure 26 displays the exterior of the images of the unit disc under the mappings
$\dfrac1{\sum_{j=1}^5a_j^0z^j }$ and $\dfrac1{\sum_{j=1}^5a_j^\epsilon z^j }$  where $ \varepsilon=0.005$ (Fig. \ref{Fig26}).

\begin{figure}[ht]
\centering
%\begin{minipage}[b]{0.45\linewidth}
\hspace{2cm}\includegraphics[scale=0.15]{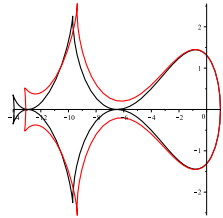}
\caption{ The illustration of the $\varepsilon$-trick (with $n=5,$ $ \varepsilon=0.005$)  }
\label{Fig26}
%\end{minipage}
\end{figure}

% Fig. \ref{Fig26}:  \\

It makes sense to use the formulas \eqref{30} of the variation of the extremal coefficients in the case of 2-cycle stabilization as well.

\section{Conclusion} In this article, the actual problem of stabilization of not
known in advance  unstable periodic orbits of discrete chaotic systems is considered. The approach
for solving the stabilization problem is based on the use of nonlinear delayed
feedback control (DFC), proposed in \cite{6} to stabilize an equilibrium and
modified in \cite{Mo} for cycles (there for the case $n=2$). 
Using non-linear (non-linearity must be related to the original
system ) DFC  is the only way to avoid restrictions on the size of the area of
possible localization of multipliers. For example, in \cite{U} it is shown that the linear DFC whith $n = 2$
is applicable only in the case of localization of multipliers in the region $[-3,1)$. It is shown in \cite{TD}
that increasing $n$ does not provide substantial advantages. Namely, if the diameter of the
areas of possible localization of multipliers are greater than 16 in the general case, or
greater than 4 in the case of a simply connected region, then the linear DFC does not solve
stabilization problem.

As it is shown in this article, the problem of limitation of the size of the region of 
multipliers location can be circumvented by using a nonlinear  DFC
with multiple delays , i.e., when $n > 2.$ For $T = 1,2$ the strength 
coefficients are found explicitly. The suggested  control stabilizes all equilibriums and 2-cycles with
negative multipliers. It is discovered  an unexpected connection between the problem of
determining the optimal coefficients and the classical Fourier analysis as well as with 
geometric complex analysis.

The problem of investigation of the stability of cycles of the length greater than 2 is reduced to the study
of stability of cycles of shorter lengths (including equilibria) for an auxiliary control systems. 
Aswas mentioned in \cite{6}, nonlinear DFC  indeed has robust properties. In the context of the current 
problems it means that one control stabilizes  all equilibriums
(or 2 - cycles) with multipliers in $[-\mu^*,1)$  where the value of $\mu^*$ depends on the number $n$
of coefficients  in delayed feedback loop. Theorems 2 and 3 implies the asymptotic dependence on the linear size 
of the region of location of multiples from $n$ at $T=1,2$ and that asymptotic is $\mu^*\sim n^2.$

A. Solyanik \cite{18} claims that this asymptotic dependence is still valid for an arbitrary $T.$ He also claim
an algorithm for constructing the coefficients of direct control, stabilizing cycles of arbitrary length.

Other generalizations of the proposed scheme are possible. In particular, the point of interest is a 
construction of direct stabilizing controls for multidimensional mappings. Preliminary investigations
indicates that the strength coefficients in these tasks  are also associated with Fej\'er polynomial coefficients .

\section{Acknowledgement} The authors would like to thank  Alexey Solyanik for valuable comments and deep insight into considered problems  and to Paul Hagelstein for interesting discussions and for the help in preparation of manuscript.

\smallskip
\bigskip

Dmitriy Dmitrishin, Odessa National Polytechnic University, 1 Shevchenko Ave., Odessa 65044, 
Ukraine. E-mail: dmitrishin@opu.ua\\
 
 Anatolii Korenovskyi, Odessa National University, Dvoryanskaya 2, Odessa 65000, Ukraine. E-mail: anakor@paco.net\\
      
 Alex Stokolos and Anna Khamitova, Georgia Southern University, Statesboro, GA 30458, USA. E-mail: astokolos@georgiasouthern.edu, \\
anna\_khamitova@georgiasouthern.edu

\end{document}